\documentclass [smallcondensed] {svjour3} 
\usepackage[utf8]{inputenc}
\usepackage{a4wide}
\usepackage{amssymb}
\usepackage{color}
\usepackage{cite}
\usepackage{booktabs}
\usepackage{mathdots}

\usepackage{tikz}
\usepackage{enumerate}
\usepackage{mathtools}
\usepackage{multirow}
\usepackage{hhline}

\newcommand{\RR}{\mathbb{R}}
\newcommand{\CC}{\mathbb{C}}
\newcommand{\FF}{\mathbb{F}}

\newcommand{\sym}{\mathbb{S}}
\newcommand{\cC}{{\mathfrak{C}}}

\newcommand{\cG}{{\cal G}}

\newcommand{\wt}{\widetilde}

\newcommand{\iunit}{{\mathfrak i}}

\newcommand{\la}{\lambda}

\newcommand{\colb}{\color{black}}

\newcommand{\rank}{{\rm rank\,}}
\newcommand{\mycomment}[1]{}

\catcode`@=11     
\font\tenex=cmex10 
\newdimen\p@renwd
\setbox0=\hbox{\tenex B} \p@renwd=\wd0 
\def\bmat#1{\begingroup \m@th
  \setbox\z@\vbox{\def\cr{\crcr\noalign{\kern2\p@\global\let\cr\endline}}%
    \ialign{$##$\hfil\kern2\p@\kern\p@renwd&\thinspace\hfil$##$\hfil
      &&\quad\hfil$##$\hfil\crcr
      \omit\strut\hfil\crcr\noalign{\kern-\baselineskip}%
      #1\crcr\omit\strut\cr}}%
  \setbox\tw@\vbox{\unvcopy\z@\global\setbox\@ne\lastbox}%
  \setbox\tw@\hbox{\unhbox\@ne\unskip\global\setbox\@ne\lastbox}%
  \setbox\tw@\hbox{$\kern\wd\@ne\kern-\p@renwd\left[\kern-\wd\@ne
    \global\setbox\@ne\vbox{\box\@ne\kern2\p@}%
    \vcenter{\kern-\ht\@ne\unvbox\z@\kern-\baselineskip}\,\right]$}%
  \null\;\vbox{\kern\ht\@ne\box\tw@}\endgroup}
\catcode`@=12    

\newcounter{algo}[section]

\renewcommand{\leq}{\leqslant}
\renewcommand{\geq}{\geqslant}

\renewcommand{\H}{\mathbb H}
\renewcommand{\theta}{\vartheta}

\definecolor{green}{rgb}{0.0,0.7,0.0}
\newcommand{\cm}[1]{{\color{black} #1}}
\newcommand{\vm}[1]{{\color{black} #1}}

\DeclareMathOperator{\diag}{diag}
\DeclareMathOperator{\rev}{rev}
\newcommand {\mat}  [1] {\left[\begin{array}{#1}}
\newcommand {\rix}      {\end{array}\right]}

\newcommand{\C}{\mathbb{C}}

\begin{document}

\title{Low rank perturbation of regular matrix pencils with symmetry structures}

\author{Fernando De Ter\'{a}n\and Christian Mehl\and Volker Mehrmann}
\institute{Fernando De Ter\'{a}n\at Universidad Carlos III de Madrid, Departamento de Matem\'{a}ticas, Legan\'{e}s, Madrid, Spain,
\email{fteran@math.uc3m.es}
\and
Christian Mehl\and Volker Mehrmann\at Technische Universit\"at Berlin, Institut f\"ur Mathematik, Sekretariat MA 4-5, 10623 Berlin,
\email{$\{$mehl,mehrmann$\}$@math.tu-berlin.de}
}

\titlerunning{Low rank perturbation of regular pencils}
\date{\today}

\maketitle

\begin{abstract}
The generic change of the Weierstra\ss\ Canonical Form  of regular complex structured matrix pencils under
\cm{generic structure-preserving} additive low-rank perturbations is studied. Several different \cm{symmetry structures} are considered
and it is shown that for most of the structures, the generic change in the eigenvalues is analogous
\cm{to the case of generic perturbations that ignore the structure}.
However, for some odd/even and palindromic structures, there is a different behavior for the eigenvalues $0$ and $\infty$,
respectively $+1$ and $-1$. The differences arise in those cases where the parity of the partial multiplicities in the perturbed
pencil provided by the generic behavior in the \cm{general structure-ignoring} case is not in accordance with the restrictions
imposed by the structure. \cm{The} new results extend results for the rank-$1$ \cm{and rank-$2$ cases that were obtained in
\cite{batzke14,batzke16} for the case of special structure-preserving perturbations. As the main tool, we use decompositions of
matrix pencils with symmetry structure into sums of rank-one pencils, as those allow a parametrization of the set of matrix pencils
with a given symmetry structure and a given rank.}
\end{abstract}

\bigskip

\section{Introduction}

The generic change in the Jordan structure of matrices under low-rank perturbations has been
established in \cite{HorM94} and was rediscovered later independently in \cite{MorD03,Sav03,Sav04}:
if a matrix $A\in\mathbb C^{n\times n}$ has an eigenvalue $\lambda_0$ with partial multiplicities
$n_1\geq\cdots\geq n_g$ (i.e., these are the sizes of the Jordan blocks associated with $\lambda_0$ in
the Jordan canonical form of $A$), then a generic perturbation of rank $r<g$ has the effect that the
perturbed matrix still has the eigenvalue $\lambda_0$ with partial multiplicities $n_{r+1}\geq\cdots\geq n_g$,
while $\lambda_0$ is no longer an eigenvalue of the perturbed matrix if a generic perturbation of
rank $r\geq g$ is applied.

Starting with \cite{MehMRR11} a series of papers has studied the generic changes in
the Jordan structure of matrices with symmetry structures under structure-preserving low-rank perturbations
and it has been observed that sometimes the behavior differs from the one under arbitrary low-rank
perturbations due to restrictions in the possible Jordan structures of the matrices with symmetry
structures, see \cite{BatMRR16,FouGJR13,Jan12,MehMRR11,MehMRR12,MehMRR13,MehMRR14,MehMRR16,MehR17}.

There are many applications where low-rank perturbations of matrix pencils \cm{with or without symmetry structures}
arise. \cm{For example,} matrix pencils are the coefficient
representations of linear differential-algebraic equations, see e.g. \cite{BreCP96,KunM06} and the references therein.
Structured low-rank perturbations are \cm{then} common when power networks or electrical circuits are considered, and the stability is
studied when interconnections are interrupted  \cite{HamIRP10,AlbAN04,DuLM13,Pet08}. These are typically  perturbations of rank one or two.
Another class of problems where the \cm{perturbations are of} low-rank compared to the system size, but not low-rank in absolute terms, are
switched systems which change their states, see e.g. \cite{GraMQSW16,HamM08,Lib03,LibT12,MehW09}.
We will study low-rank perturbations of structured matrix pencils from an abstract \cm{matrix-theoretical} point of view and do not
consider the many concrete applications where this topic has major implications.

A result on the generic
change of the Weierstra{\ss} structure (namely, the partial multiplicities) under low-rank perturbations
of regular pencils without any
additional symmetry structures has been established as early as in \cite{ddm}, where \emph{genericity}
was understood in the following sense: a subset of a finite-dimensional linear space of perturbations is called
generic if it is an open dense subset with respect to the natural topology on the linear space.
In contrast to this notion, a stronger concept of genericity had been used in the references \cm{starting from \cite{MehMRR11}}:
in that sense, a subset $\cG$ of $\mathbb C^m$ is generic if its complement $\mathbb C^m\setminus\cG$
is contained in a proper algebraic set, i.e., a set of common zeros of finitely many polynomials in $m$ variables
that does not coincide with the full set $\mathbb C^m$. The latter concept is not only stronger {\colb than} the previous
one (clearly any generic set in the latter sense is an open dense subset of $\mathbb C^m$ while the converse is
not true in general), but it also allowed an easy transition from the complex to the real case as it was
shown in \cite{MehMRR13}. {\colb This} concept requires the parametrization of the set of
considered perturbations as a subset of $\mathbb C^m$. In \cite{dd16} such a parametrization of the set
of pencils of rank at most $r$ was introduced and the result from \cite{ddm} could be generalized to the
stronger concept of genericity in the sense of its complement being contained in a proper algebraic set.
The main result obtained in \cite{dd16} states that the generic behavior in the case of matrix pencils coincides
with the one for matrices. More precisely, if $A+\la B$ is a regular pencil and $\la_0\in\CC\cup\{\infty\}$ is an
eigenvalue of $A+\la B$ with partial multiplicities $n_1\geq\cdots\geq n_g$, then a generic additive perturbation
of $A+\la B$ with rank $r$ ``destroys" the $r$ largest multiplicities, so that the perturbed pencil
has \cm{the partial multiplicities $n_{r+1}\geq\cdots\geq n_g$} at $\la_0$.

Surprisingly, the case of matrix pencils with some additional symmetry structure has not yet been as well studied as the matrix case.
The first attempt to investigate the generic change in the Weierstra{\ss} structure of such matrix pencils
under structure-preserving low-rank perturbations was undertaken in \cite{batzke14,batzke16,batzke-thesis},
where the \cm{cases} of rank-$1$ perturbations and special perturbations of rank two \cm{were} considered - the restriction
to these cases was due to the fact that straightforward parameterizations were available in that case.
While it was shown in \cite{BatMRR16} how the knowledge of the behavior in the rank-one case can be extended
to arbitrary rank in the matrix case, a similar transition is not possible in the pencil case, since a structured
pencil of small rank can in general not be written as a sum of those rank-$1$ or rank-$2$ pencils that were considered in
\cite{batzke14,batzke16,batzke-thesis}. Therefore, the case of structure-preserving perturbations of rank larger than
two remained an open problem.

It is our aim to fill this gap by extending the ideas from \cite{dd16} to develop parameterizations of
low-rank pencils with symmetry structures and obtain results on the generic change in the Weierstra{\ss}
structure of structured matrix pencils under low-rank structure-preserving perturbations. Moreover, we will
also consider one aspect that has not been considered in the pencil case so far: the generic multiplicity of
newly generated eigenvalues.

Low-rank perturbation of singular matrix pencils has been considered in \cite{dd07}, restricted to the case
where the perturbed pencil remains singular. A different generic behavior on the change of the partial multiplicities
of eigenvalues is shown in this case. In particular, for generic perturbations, all partial multiplicities of any
eigenvalue of the unperturbed pencil stay after perturbation. In this paper, however, we restrict ourselves to
regular matrix \cm{pencils} which remain regular after perturbation (which is a generic condition). Nonetheless,
singular pencils naturally appear in the context of the present work, since low-rank pencils are necessarily singular.

{\colb The  paper is organized as follows. In Section \ref{notation.sec} we introduce some notation and recall the
Weierstra{\ss} canonical form. The symmetry structures considered in the paper are introduced in Section \ref{rank1.sec},
where we also present the rank-$1$ decomposition of low-rank structured pencils for any of these structures. We consider
the Hermitian and $\top$-even cases in full detail, and from the results for these two structures we derive the results
for the remaining symmetry structures. Section \ref{main.sec} contains the main results of the paper, namely the description
of the generic change of the partial multiplicities of regular pencils \cm{with symmetry structures} under low-rank \cm{structure}-preserving
perturbations. If we restrict ourselves to pencils with real entries, the approach followed in the manuscript is no longer valid.
\cm{In the short Section \ref{real.sec} we briefly discuss the case of real matrix pencils with symmetry structures and
explain why the results of the previous sections cannot be applied in that case.}
In Section \ref{conclusion.sec} we summarize the contributions of the
paper and we present some lines of further research. Appendix \ref{appendix} contains the proof of a couple of technical results
used in Section \ref{main.sec}.}

\section{Notation and basic results}\label{notation.sec}

By $e_i$ we denote the $i$th canonical vector of appropriate size, i.e., the $i$th column of the identity
matrix with the appropriate order. By $\iunit$ we denote the imaginary unit. The notation $0_{m\times n}$
stands for the $m\times n$ zero matrix. When either $m=1$ or $n=1$, then we just write $0_n$ or $0_m$,
respectively. Note that we use the same notation for zero rows and zero columns, but which is the right one
is clear by the context.

As usual, $\CC^{m\times n}$ denotes the set of $m\times n$ matrices with complex entries, and $\CC^n$ denotes
the set of vectors with $n$ complex coordinates in column form (i.e., $\CC^n=\CC^{n\times1}$). Given a
matrix $A\in\CC^{m\times n}$, we denote by $A(i,j)$ the $(i,j)$ entry of $A$. By $\C[\la]^n$ we denote the
set of vector polynomials with $n$ coordinates, i.e., the set of vectors with $n$ coordinates which are
polynomials in the variable $\la$.

We use $L(\la)$ for general pencils, as well as for the given (unperturbed) pencil, whereas $E(\la)$ will
be used for the perturbation pencil.
The notation $\star$ is used for either the transpose ($\top$) or the conjugate transpose ($*$) of a matrix.
Given a matrix pencil $L(\la)=A+\lambda B$ (or just $L$, for short), by $L(\la)^\star$ (or $L^\star$, for short) we
denote the pencil $A^\star+\la B^\star$. It is important to note that, when $\star=*$, then the operator
$*$ does not affect the variable $\la$, but just the coefficients of the pencil. The pencil is said to be {\em regular}
if it is square and $\det L(\la)$ is not identically zero. Otherwise, it is said to be {\em singular}.
The {\em rank} of $L(\la)$, denoted $\rank L$, is the size of the largest non-identically
zero minor of $L(\la)$ (considering the minors as polynomials in $\la$), i.e., the rank of $L(\la)$
considered as a matrix over the field of rational functions in $\la$. In other words, it is the quantity
$
\max_{\lambda\in\mathbb C}\operatorname{rank}(A+\lambda B).
$
This is sometimes referred to as the
{\em normal rank} in the literature (see, for instance, \cite{eek2}). Note that, if $A+\lambda B$ is a
square $n\times n$ matrix pencil with rank $r<n$, then $A+\lambda B$ is singular.

The {\em reversal} $\rev(A+\lambda B)$ of a matrix pencil $A+\la B$ is the matrix pencil $B+\la A$.

By $L_\alpha$ we denote a {\em right singular block of order $\alpha$}, i.e., the $\alpha\times(\alpha+1)$
pencil
\[
L_\alpha:=\left[\begin{array}{cccc}
\la&1&&\\
&\ddots&\ddots&\\
&&\la&1
\end{array}\right]_{\alpha\times(\alpha+1)}.
\]

By $J_k(a-\la)$ we denote a pencil corresponding to a $k\times k$ Jordan block associated with the eigenvalue $a$, namely
\[
J_k(a-\la):=\left[\begin{array}{cccc}a-\la&1&\\&\ddots&\ddots\\&&a-\la&1\\&&&a-\la\end{array}\right]_{k\times k},
\]
and $R$ denotes the {\em reverse identity} matrix, namely
\[
R:=\left[\begin{array}{ccc}&&1\\&\iddots&\\1&&\end{array}\right],
\]
where the size will be clear by the context.

\begin{remark}\label{reversal.rem}
If $w\in\CC[\la]^n$ is a vector polynomial of degree (at most) $1$, and $v\in\CC^n$ (i.e., a constant vector)
then $\rev (vw^\star)=v\cdot(\rev w)^\star$.
\end{remark}

If $A+\lambda B$ is a regular $n\times n$ matrix pencil,
then it can be transformed to Weierstra{\ss}
canonical form (WCF). More precisely, there exist nonsingular matrices $S,T\in\mathbb C^{n\times n}$ such that
\begin{eqnarray*}
S(A+\lambda B)T&=&\diag\big(\mathcal J_{n_{1,1}}(a_1-\lambda),\dots,\mathcal J_{n_{1,g_1}}(a_1-\lambda),\dots,
\mathcal J_{n_{\kappa,1}}(a_\kappa-\lambda),\dots,\mathcal J_{n_{\kappa,g_\kappa}}(a_\kappa-\lambda),\\
&&\textcolor{white}{\diag\big(}\rev J_{n_{\kappa+1,1}}(-\lambda),\dots,\rev J_{n_{\kappa+1,g_{\kappa+1}}}(-\lambda)\big).
\end{eqnarray*}
Here $\kappa\in\mathbb N$, and $a_1,\dots,a_\kappa\in\mathbb C$ are the finite eigenvalues of $A+\lambda B$ with
geometric multiplicities $g_1,\dots,g_\kappa$, respectively. The value $g_{\kappa+1}$
is the geometric multiplicity of the infinite eigenvalue, where we allow $g_{\kappa+1}=0$ for the case that $\infty$ is
not an eigenvalue of the pencil. The parameters $n_{i,1},\dots,n_{i,g_i}$ are
called the \emph{partial multiplicities} of $A+\lambda B$ at $\lambda_i$. Without loss of generality, we may
assume that they are ordered non-increasingly, i.e., we have $n_{i,1}\geq\cdots\geq n_{i,g_i}$.

If $A+\lambda B$ is a singular $m\times n$ matrix pencil, then the
corresponding canonical form is the Kronecker canonical form (KCF): there exist nonsingular matrices $S\in\mathbb C^{m\times m}$
and $T\in\mathbb C^{n\times n}$ such that
\[
S(A+\lambda B)T=\diag\big(\widetilde L(\lambda),L_{\alpha_1},\dots,L_{\alpha_\eta},L_{\beta_1}^\top,\dots,L_{\beta_\xi}^\top\big)
\]
\vm{with $\widetilde L(\lambda)$ in WCF.}
Here, the parameters $\alpha_1,\dots,\alpha_\eta\in\mathbb N$ and $\beta_1,\dots,\beta_\xi\in\mathbb N$ are called the
{\em right} or {\em left minimal indices}, respectively.

\section{Representation of structured pencils as a sum of rank-$1$ pencils}\label{rank1.sec}


It is well-known, see e.~g. \cite{Gan59a}, that any Hermitian or symmetric matrix $A\in\CC^{n\times n}$ with $\rank A=r\leq n$
can be written as a sum of rank-$1$ matrices of the same structure (this is an immediate consequence of
the so-called {\em spectral decomposition}). In particular, if $A$ is  symmetric, then it can be written as
$
A= u_1u_1^\top+\cdots+ u_ru_r^\top
$
(or $A=s_1 u_1u_1^\top+\cdots+s_ru_ru_r^\top$ if we restrict ourselves
to real coefficients), whereas if $A$ is Hermitian, then it can be written as
$
A=s_1 u_1u_1^*+\cdots+s_ru_ru_r^*
$
where $s_1,\dots,s_r\in\{+1,-1\}$ are \emph{signs}. By \emph{Sylvester's Law of Inertia}, the numbers of
positive (resp. negative) signs among $s_1,\dots,s_r$ are uniquely determined.

It is natural to ask whether an analogous decomposition holds for matrix pencils \cm{with symmetry structures}.
The structures we are interested in are compiled in the following list. A matrix pencil $L(\la)=A+\la B$ with
$A,B\in\mathbb C^{n\times n}$ is said to be

\begin{itemize}
\item {\it Hermitian} if $A=A^*,B=B^*$;
\item {\it symmetric} if $A=A^\top,B=B^\top$;
\item {\it skew-Hermitian} if $A^*=-A,B^*=-B$;
\item {\it skew-symmetric} if $A^\top=-A,B^\top=-B$;
\item {\it $\star$-even} if $A^\star=A,B^\star=-B$;
\item {\it $\star$-odd} if $A^\star=-A,B^\star=B$;
\item {\it $\star$-palindromic} if $A^\star=B$;
\item {\it $\star$-anti-palindromic} if $A^\star=-B$.
\end{itemize}

The name {\em $\star$-alternating} is also used as an umbrella term for both $\star$-even and $\star$-odd.

For the sake of brevity, we will use the following notation for the set of $n\times n$ structured matrix
pencils with rank at most $r$, for each of the previous structures:
\begin{center}
\begin{tabular}{c|c}
structure&notation\\\hline
Hermitian&$\H_r$\\
symmetric&$Sym_r$\\
skew-Hermitian&$S\H_r$\\
skew-symmetric&$SSym_r$\\
$\star$-even&$Even_r^\star$\\
$\star$-odd&$Odd_r^\star$\\
$\star$-palindromic&$ Pal_r^\star$\\
$\star$-anti-palindromic&$Apal_r^\star$
\end{tabular}
\end{center}
Note that, for the ease of notation, and since all matrices considered in this paper are of the same size
$n\times n$, there is no explicit mention of the size in the notation introduced above.

We start by showing the existence of a decomposition of structured low-rank pencils as a sum of structured
rank-$1$ pencils. For this, we will use structured canonical forms for these kinds of pencils. These canonical
forms comprise the information displayed in the WCF, with the appropriate restrictions imposed by the corresponding
symmetry structure. We refer
to \cite{batzke-thesis} for these canonical forms, since they are all gathered in this reference, even
though all of them were introduced in earlier references. Furthermore, we focus on the case of Hermitian
pencils and will give a detailed proof for this case only, while for the cases of other structures we will
either reduce them to the Hermitian case or mention in which parts the proofs of the corresponding results
differ from the Hermitian case.

\subsection{Rank-$1$ decompositions for the Hermitian case}

First, we recall the well-known canonical form for Hermitian pencils under congruence, see, e.g.,
\cite[Theorem 2.20]{batzke-thesis}.
\begin{theorem}\label{thm:canformherm} {\rm(Canonical form of Hermitian pencils).}
Let $E(\lambda)$ be a Hermitian $n\times n$ matrix pencil. Then there exists a nonsingular matrix $P$ such that
\[
P^*E(\lambda)P=\operatorname{diag}\big(E_1(\lambda),\dots,E_m(\lambda)\big),
\]
where each pencil $E_j(\lambda)$, for $j=1,\dots,m$, has exactly one of the following four forms:
\begin{itemize}
\item[{\rm i)}] blocks $\sigma RJ_k(a-\lambda)$ associated with a real eigenvalue $a\in\mathbb R$ and a
sign $\sigma\in\{+1,-1\}$;
\item[{\rm ii)}] blocks
\[
\rev\big(\sigma RJ_k(-\lambda)\big)=\sigma\mat{cccc}&&&-1\\ &&-1&\lambda\\ &\iddots&\iddots&\\ -1&\lambda&&\rix
\]
associated with the eigenvalue infinity and a sign $\sigma\in\{+1,-1\}$;
\item[{\rm iii)}] blocks $R\diag\big(J_k(\overline{\mu}-\lambda),J_k(\mu-\lambda)\big)$ associated with
a pair $(\mu,\overline{\mu})$ of conjugate complex eigenvalues, with $\mu\in\mathbb C$ having positive
imaginary part;
\item[{\rm iv)}] blocks
\[
\mat{cc} 0&L_k^\top\\ L_k&0\rix
\]
consisting of a pair of one right and one left singular block with the same index $k$.
\end{itemize}
The parameters $a,k,\sigma$, and $\mu$ depend on the particular block $L_j(\lambda)$ and may be
distinct in different blocks. Furthermore, the canonical form is unique up to permutation of blocks.
\end{theorem}
The signs $\sigma$ in the blocks of type i) and ii) in Theorem~\ref{thm:canformherm} are invariant
under congruence transformations and their collection is referred to as the \emph{sign characteristic}
of the Hermitian pencil following the terminology of \cite{GohLR05,MehNTX16}. The following result
presents a decomposition of a given Hermitian pencil as a sum of rank-$1$ Hermitian pencils, which extends the one for
Hermitian matrices mentioned at the beginning of this section. \vm{Hereafter}, we deal with polynomial vectors,
namely vectors $v(\la)\in\CC[\la]^n$, though, for brevity, in general we will drop the dependence on $\la$.
For a given $v(\la)\in\CC[\la]^n$, by $\deg v$ we denote the largest degree of the entries of $v$. In order to avoid confusion,
it is important to recall that, given a pencil $A+\la B$, we write $(A+\la B)^*$ to denote the pencil $A^*+\la B^*$,
i.e., we only apply the conjugate transpose to the coefficients of the pencil, and not to the variable $\la$.
\begin{theorem}\label{rank1-herm.th} {\rm(Rank-$1$ decomposition for Hermitian pencils).}
If $E(\la)$ is a Hermitian $n\times n$ matrix pencil with $\rank E=r\leq n$, then it can be written as
\begin{equation}\label{herm-rank1}
E(\la)=(a_1+\la b_1)u_1u_1^*+\cdots+(a_\ell+\la b_\ell)u_\ell u_\ell^*+v_1w_1^*+\cdots+v_sw_s^*
+w_1v_1^*+\cdots+w_sv_s^*,
\end{equation}
where $a_i,b_i\in\RR$, for $i=1,\hdots,\ell$, and
\begin{itemize}
\item[(i)] $\ell+2s=r$,
\item[(ii)] $\deg u_1=\cdots=\deg u_\ell=0=\deg v_1=\cdots=\deg v_s$ and $\,\deg w_1,\hdots,\deg w_s\leq 1$.
\end{itemize}
\end{theorem}
\begin{proof}
It suffices to prove the statement for $E(\la)$ being in Hermitian canonical form as in
Theorem~\ref{thm:canformherm}. To see this, just notice that if $K_E(\la)$ is the Hermitian canonical
form of $E(\la)$ and \cm{if} it has a decomposition
\[
K_E(\la)=(a_1+\la b_1)\widetilde u_1\widetilde u_1^*+\cdots+(a_\ell+\la b_\ell)\widetilde u_\ell\wt u_\ell^*
+\wt v_1\wt w_1^*+\cdots+\widetilde v_s\widetilde w_s^*+\wt w_1\wt v_1^*+\cdots+\wt w_s\wt v_s^*,
\]
\vm{as in \eqref{herm-rank1},} then there exists a nonsingular matrix $P$ such that
\[
\begin{array}{ccl}
E(\la)=PK_E(\la)P^*&=&(a_1+\la b_1)u_1u_1^*+\cdots+(a_\ell+\la b_\ell)u_\ell u_\ell^*\\
&&+v_1w_1^*+\cdots+v_sw_s^*+w_1v_1^*+\cdots+w_sv_s^*
\end{array}
\]
with $u_i=P\wt u_i, v_j=P\wt v_j,$ and $w_j=P\wt w_j$, for $i=1,\hdots,\ell$ and $j=1,\hdots,s$.
This gives the desired decomposition \eqref{herm-rank1} for $L(\la)$.

So we may assume $E(\la)$ to be in Hermitian canonical form, which is a direct sum of blocks of the four
different types i)--iv) as in Theorem~\ref{thm:canformherm}. We will provide a decomposition
like~\eqref{herm-rank1} for each of these blocks.

1) A $k\times k$ block associated with a real eigenvalue $a\in\mathbb R$ and sign $\sigma\in\{+1,-1\}$
can be decomposed as follows, depending on whether $k$ is odd or even.
If $k$ is even then
\begin{eqnarray*}
&&\sigma RJ_k(a-\la)\\
&=&\footnotesize\sigma\left[\begin{array}{ccccccc} &&&&&&a-\la\\&&&&&0&1/2\\&&&&a-\la&0\\ &&&0&1/2&&\\
&&\iddots&\iddots&&\\&a-\la&0&&\\0&1/2\end{array}\right]
+\sigma\left[\begin{array}{ccccccc} &&&&&&0\\&&&&&a-\la&1/2\\&&&&0&0\\ &&&a-\la&1/2&&\\
&&\iddots&\iddots&&\\&0&0&&\\a-\la&1/2\end{array}\right]\\
&=&\textcolor{white}{+}\,\sigma\left(\left[\begin{array}{c}a-\la\\1/2\\0_{k-2}\end{array}\right]e_k^*+
\left[\begin{array}{c}0_2\\a-\la\\1/2\\0_{k-4}\end{array}\right]e_{k-2}^*+\cdots+
\left[\begin{array}{c}0_{k-2}\\a-\la\\1/2\end{array}\right] e_2^*\right)\\
&&+\,\sigma\left(e_k\left[\begin{array}{c}a-\la\\1/2\\0_{k-2}\end{array}\right]^*
+e_{k-2}\left[\begin{array}{c}0_2\\a-\la\\1/2\\0_{k-4}\end{array}\right]^*
+\cdots+e_2\left[\begin{array}{c}0_{k-2}\\a-\la\\1/2\end{array}\right]^*\right),
\end{eqnarray*}
which is of the form \eqref{herm-rank1} with $v_i=\sigma e_{2i}$ and
$w_i=\left[\begin{array}{cccc}0_{k-2i}&a-\la&1/2&0_{2i-2}\end{array}\right]^*$, for $i=1,\hdots,k/2$.
Note that $\sigma$ can be included either in $v_i$ or $w_i$, for $i=1,\hdots,k/2$.

If $k$ is odd, then we can split the block in two pieces
\begin{eqnarray*}
&&\sigma RJ_k(a-\la)\\
&=&\sigma(a-\la) e_{\frac{k+1}{2}}e^*_{\frac{k+1}{2}}
+\sigma\left[\begin{array}{cccc|ccc} &&&&&&a-\la\\&&&&&\iddots&1\\&&&&a-\la&\iddots&\\ &&&0&1\\\hline &&a-\la&1\\
&\iddots&\iddots&&\\a-\la&1&&&&\end{array}\right]\\
&=&\sigma(a-\la) e_{\frac{k+1}{2}}e^*_{\frac{k+1}{2}}+\sigma\left(\left[\begin{array}{c}a-\la\\1\\
0_{k-2}\end{array}\right]e_k^*+\left[\begin{array}{c}0\\a-\la\\1\\0_{k-3}\end{array}\right]e_{k-1}^*+\cdots+
\left[\begin{array}{c}0_{\frac{k-3}{2}}\\a-\la\\1\\0_{\frac{k-1}{2}}\end{array}\right]e_{\frac{k-1}{2}}^*\right)\\
&&\textcolor{white}{\sigma(a-\la) e_{\frac{k+1}{2}}e^*_{\frac{k+1}{2}}}
+\sigma\left(e_k\left[\begin{array}{c}a-\la\\1\\0_{k-2}\end{array}\right]^*+e_{k-1}\left[\begin{array}{c}0\\
a-\la\\1\\0_{k-3}\end{array}\right]^*+\cdots+
e_{\frac{k-1}{2}}\left[\begin{array}{c}0_{\frac{k-3}{2}}\\a-\la\\1\\0_{\frac{k-1}{2}}\end{array}\right]^*\right),
\end{eqnarray*}
and proceed as in the previous case with the last two summands.

2) A $k\times k$ block associated with $\infty$ and sign characteristic $\sigma$ can be decomposed in a
similar way, replacing the roles of \cm{$a-\la$} and $1$ in the previous case by \cm{$-1$} and $\la$, respectively.

3) A pair of $k\times k$ blocks corresponding to a pair of complex conjugate eigenvalues $\mu,\overline\mu$
can be decomposed as
\begin{eqnarray*}
R\diag(J_k(\overline\mu-\lambda),J_k(\mu-\lambda))
&=&\left[\begin{array}{cccc|cccc}&&&&&&&\mu-\lambda\\&&&&&&\iddots&1\\&&&&&\mu-\lambda&\iddots\\&&&&\mu-\lambda&1 \\\hline
&&&\overline\mu-\lambda&&&&\\&&\iddots&1&&&&\\&\overline\mu-\lambda&\iddots&&&&&\\\overline\mu-\lambda&1&&&&&&
\end{array}\right]\\
&=&\textcolor{white}{+}\left[\begin{array}{c}\mu-\lambda\\1\\0_{2k-2}\end{array}\right]e_{2k}^*
+\cdots+\left[\begin{array}{c}0_{k-2}\\\mu-\lambda\\1\\{0_k}\end{array}\right]e_{k+2}^*
+\left[\begin{array}{c}0_{k-1}\\\mu-\lambda\\{0_k}\end{array}\right]e_{k+1}^*\\
&&+e_{k+1}\left[\begin{array}{c}0_{k-1}\\\mu-\lambda\\{0_k}\end{array}\right]^*
+e_{k+1}\left[\begin{array}{c}0_{k-2}\\\mu-\lambda\\1\\{ 0_k}\end{array}\right]^*
+\cdots+e_{2k}\left[\begin{array}{c}\mu-\lambda\\1\\0_{ 2k-2}\end{array}\right]^*,
\end{eqnarray*}
which is of the desired form.

4) Finally, a pair consisting of a left and a right singular block with respective sizes $k\times(k+1)$
and $(k+1)\times k$ can be decomposed as
\[
\left[\begin{array}{cc}0&L_k^\top\\L_k&0 \end{array}\right]=
e_{k+1}\left[\begin{array}{c}\la\\1\\0_{2k-1}\end{array}\right]^*+\cdots+e_{2k+1}\left[\begin{array}{c}0_{k-1}\\
\la\\1\\0_{k}\end{array}\right]^*
+\left[\begin{array}{c}\la\\1\\0_{2k-1}\end{array}\right]e_{k+1}^*+\cdots+\left[\begin{array}{c}0_{k-1}\\
\la\\1\\0_k\end{array}\right]e_{2k+1}^*,
\]
which is, again, in the desired form.

So each block in the canonical form has a decomposition like~\eqref{herm-rank1}. Forming this direct sum
by padding up with zeroes in the entries of each vector corresponding to the other blocks, we arrive at a
decomposition~\eqref{herm-rank1} for $E(\la)$ given in Hermitian canonical form.
\end{proof}

\begin{remark}\label{rem:concise}
Note that $u_1,\hdots,u_\ell$ and $v_1,\hdots,v_s$ are constant vectors, but $w_1,\hdots,w_s$ are (column)
pencils, which means that their entries are polynomials in $\la$ with degree at most $1$.
Thus writing $w_i(\lambda)=w_{i,A}+\lambda w_{i,B}$ for $i=0,\dots,s$ with
$w_{1,A},\dots,w_{s;A},w_{1,B},\dots,w_{s,B}\in\mathbb C^n$ and using the notation
\begin{align*}
U:=&\mat{ccc}u_1&\dots&u_\ell\rix, &
V:=&\mat{ccc}v_1&\dots&v_s\rix,\\
W_A:=&\mat{ccc}w_{1,A}&\dots&w_{s,A}\rix,&
W_B:=&\mat{ccc}w_{1,B}&\dots&w_{s,B}\rix,\\
D_A:=&\operatorname{diag}(a_1,\dots,a_\ell),&
D_B:=&\operatorname{diag}(b_1,\dots,b_\ell),
\end{align*}
we can write~\eqref{herm-rank1} in the concise form
\begin{equation}\label{eq:concise}
E(\lambda)=U(D_A+\lambda D_B)U^*+V(W_A^*+\lambda W_B^*)+(W_A+\lambda W_B)V^*.
\end{equation}
\end{remark}

\begin{remark}\label{rem:oddblocks}
By the construction in the proof of Theorem~\ref{rank1-herm.th}, the terms of the form $(a+\la b)uu^*$ in the
decomposition~\eqref{herm-rank1} come either from blocks associated with real eigenvalues or from blocks
associated with the infinite eigenvalue, and in both cases the blocks have odd size.
\end{remark}

\begin{remark}\label{rem:linind}
If~\eqref{herm-rank1} is a decomposition into rank-$1$ pencils as in Theorem~\ref{rank1-herm.th}, then
the vectors $u_1,\dots,u_\ell$, $v_1,\dots,v_s$ are linearly independent. To see this, assume that they are linearly dependent.
Let $X:=\mat{ccc}X_1&X_2&X_3\rix\in\mathbb C^{n\times n}$ be nonsingular such
that the columns of $\mat{cc}X_1&X_2\rix\in\mathbb C^{n\times(p+q)}$ span the orthogonal complement of
the span of $v_1,\dots,v_s$ and the columns of $X_1\in\mathbb C^{n\times p}$ span the orthogonal complement
of the span of $u_1,\dots,u_\ell,v_1,\dots,v_s$. Then we have $ p+q\geq n-s$ and, because of the
assumed linear dependency, $ p> n-(\ell+s)$. Observe that
\[
X^*E(\lambda)X=\bmat{& \scriptstyle p & \scriptstyle q &\scriptstyle n-p-q\cr
\scriptstyle p & 0&0&X_1^*E(\lambda)X_3\cr \scriptstyle q & 0&X_2^*E(\lambda)X_2 & X_2^*E(\lambda)X_3\cr
\scriptstyle n-p-q & X_3^*E(\lambda)X_1 & X_3^*E(\lambda)X_f & X_3^*E(\lambda)X_3}
\]
 from which we obtain that the rank of $E(\lambda)$ is bounded by
\[
2(n-p-q)+q=n-p+n-p-q\cm{\,<\,} \ell+s+s=r,
\]
which is in contradiction to the assumption in Theorem~\ref{rank1-herm.th} that $E(\lambda)$ has rank $r$.
\end{remark}

Unfortunately, the decomposition~\eqref{herm-rank1} is far from being unique as the following
example illustrates.

\begin{example}\label{ex:twodecomps}\rm
Consider the Hermitian pencil
\[
E(\lambda):=\lambda\mat{cc}0&1\\ 1&0\rix-\mat{cc}0&1\\1&0\rix=\mat{cc}0&\lambda-1\\ \lambda-1&0\rix
\]
and let $u_1=\frac{1}{\sqrt{2}}\mat{cc}1&1\rix^\top$, $u_2=\frac{1}{\sqrt{2}}\mat{cc}-1&1\rix^\top$,
$v_1=\mat{cc}1&0\rix^\top$, and $w_1=\mat{cc}0&\lambda-1\rix^\top$. Then we have
\[
E(\lambda)=(\lambda-1)u_1u_1^*-(\lambda-1)u_2u_2^*=v_1w_1^*+w_1v_1^*.
\]
\end{example}

In particular, Example~\ref{ex:twodecomps} shows that also the parameters $\ell$ and $s$ from
Theorem~\ref{rank1-herm.th} are not unique, as in the first decomposition we have $\ell=2$ and $s=0$ and in the
latter we have $\ell=0$ and $s=2$. However, the values of $\ell$ and $s$ can be fixed by requiring $\ell$ to be
minimal. Interestingly, in that case the minimal parameter $\ell$ depends on the sign characteristic of the
Hermitian pencil. In order to state the following theorem, we recall the definition
of the so-called \emph{sign sum} from \cite{Meh00}.

\begin{definition}\label{def:signsum}
Let $E(\lambda)$ be a Hermitian $n\times n$ pencil and let $\mu\in\mathbb R$ be an eigenvalue of $E(\lambda)$.
Assume that $(n_1,\dots,n_m,n_{m+1},\dots,n_q)$ are the sizes of the blocks associated with the eigenvalue
$\mu$ in the Hermitian canonical form of $E(\lambda)$, where $n_1,\dots,n_m$ are odd and $n_{m+1},\dots,n_q$
are even. Furthermore, let $(\sigma_1,\dots,\sigma_m,\sigma_{m+1},\dots,\sigma_q)$ be the corresponding signs
(of the blocks associated with $\mu$) from the sign characteristic of $E(\lambda)$. Then the
\emph{signsum} $\operatorname{sigsum}(E,\mu)$ of $\mu$ is defined as
\[
\operatorname{sigsum}(E,\mu):=\sum_{j=1}^m\sigma_j.
\]
If $\infty$ is an eigenvalue of $E(\lambda)$, then the \emph{signsum}
of $\infty$ is defined as
\[
\operatorname{sigsum}(E,\infty):=\operatorname{sigsum}(\rev E,0).
\]
\end{definition}

Thus, the signsum of the real eigenvalue $\mu$ of a Hermitian pencil is just the sum of the signs that
correspond to blocks of odd size associated with $\mu$.

\begin{example}\rm
Consider the following three Hermitian pencils
\[
E_1(\lambda)=\mat{cccc}1-\lambda&0&0&0\\ 0&0&0&1-\lambda\\ 0&0&1-\lambda&1\\ 0&1-\lambda&1&0\rix,
\]
\[
E_2(\lambda)=\mat{cc}1-\lambda&0\\ 0&\lambda-1\rix,\quad
E_3(\lambda)=\mat{cc}0&1-\lambda\\ 1-\lambda&1\rix,
\]
which all have just the single eigenvalue $a=1$. Then we have
$\operatorname{sigsum}(E_1,1)=2$, since $E_1(\lambda)$ has two odd-sized blocks associated with $a=1$
(one of size one and one of size three), both having the sign $+1$. On the other hand
$\operatorname{sigsum}(E_2,1)=0$ as $E_2(\lambda)$ has two blocks of size one, but with opposite signs
$+1$ and $-1$. For the pencil $E_3(\lambda)$, we also obtain $\operatorname{sigsum}(E_3,1)=0$, because
it has no odd-sized blocks associated with the eigenvalue $a=1$, but just one block of size two.
In that case, the sum in Definition~\ref{def:signsum} is empty and thus, by definition, equal to zero.
\end{example}

\begin{theorem}\label{thm:sigsum}
Let $E(\lambda)$ be a Hermitian $n\times n$ pencil and let $\mu_1,\dots,\mu_p\in\mathbb R\cup\{\infty\}$ be the
pairwise distinct real eigenvalues of $E(\lambda)$. (Infinity is interpreted as a possible real
eigenvalue here.) Furthermore, let~\eqref{herm-rank1} as in Theorem~\mbox{\rm\ref{rank1-herm.th}} be a decomposition
of $E$ into rank-$1$ pencils so that the parameter $\ell$ from Theorem~\mbox{\rm\ref{rank1-herm.th}}
is minimal among all possible such decompositions. Then
\begin{equation}\label{eq:signsum}
\ell=\sum_{j=1}^p|\operatorname{sigsum}(E,\mu_j)|.
\end{equation}
\end{theorem}
\begin{proof}
In the following, let $\ell_0$ denote the right-hand-side of~\eqref{eq:signsum}, i.e.,
$\ell_0=\sum_{j=1}^p|\operatorname{sigsum}(E,\mu_j)|$.\\
``$\leq$'': We first show that there exists a decomposition as in~\eqref{herm-rank1} such that $\ell=\ell_0$.
Using the same construction as in the proof of Theorem~\ref{rank1-herm.th}, we see from Remark~\ref{rem:oddblocks}
that in their decomposition into rank-$1$ pencils
only blocks of odd-size that are associated with real eigenvalues (including $\infty$) have a term of the
form $(a+\lambda b)uu^*$ (with $a,b\in\mathbb R$ and $u\in\mathbb C^n$), and thus only those blocks contribute to the
number $\ell$ in the decomposition~\eqref{herm-rank1}.
Therefore and because it is sufficient to consider each real eigenvalue separately, we may assume,
without loss of generality, that $E(\lambda)$ is regular and
only has a single eigenvalue $\mu$ that is real \cm{and finite}, such that all blocks in the Hermitian canonical form of $E(\lambda)$
associated with $\mu$ have odd size. We then have to show that $E(\lambda)$ has a decomposition
as in~\eqref{herm-rank1} with $\ell=|\operatorname{sigsum}(E,\mu)|$.

To this end, assume that the Hermitian canonical form of the pencil $E(\lambda)$ consists of $m$ blocks
with size $n_1,\dots,n_m$ (which are all odd). Let $\sigma_1,\dots,\sigma_m$ be the signs from the sign
characteristic of $E(\lambda)$, where $\sigma_j$ is associated with $n_j$ for $j=1,\dots,m$. By the
construction in the proof of Theorem~\ref{rank1-herm.th}, we then obtain a decomposition of the form
\begin{equation}\label{eq:sigsumdec}
E(\lambda)=\sigma_1(a-\lambda)u_1u_1^*+\cdots+\sigma_m(a-\lambda)u_mu_m^*+v_1w_1^*+\cdots+v_sw_s^*+ w_1v_1^*+\cdots+ w_sv_s^*.
\end{equation}
Suppose that $m=m_++m_-$, where $m_+$ is the number of blocks with positive sign $\sigma_j$ and
$m_-$ is the number of blocks with negative sign $\sigma_j$. Then $\operatorname{sigsum}(E,a)=|m_+-m_-|$,
i.e., if we try to pair up the blocks into pairs consisting of two blocks with opposite signs (but
possibly different sizes) then
the signsum of $a$ corresponds to the number of blocks that will remain unpaired. In particular, all
of these remaining blocks will have the same sign. Thus, to prove the assertion, it remains to show that
in the decomposition~\eqref{eq:sigsumdec} each summand
\[
(a-\lambda)u_iu_i^*-(a-\lambda)u_ju_j^*
\]
(where we have $\sigma_i=1$ and $\sigma_j=-1$) can be replaced by a summand of the form
$v_kw_k^*+w_kv_k^*$ with $v_k\in\mathbb C^n$ and $w_k$ being an $n\times 1$ pencil.
This goal can be achieved by choosing $v_k=u_i+\iunit u_j$ and $w_k=\frac{1}{2}(a-\lambda)(u_i-\iunit u_j)$.

``$\geq$'': It remains to show that $\ell$ cannot be chosen smaller than $\ell_0$.
Thus, let~\eqref{herm-rank1} be a decomposition of $E(\lambda)$ into rank-1-pencils with some
$\ell<\ell_0$. By Remark~\ref{rem:linind}, the columns of the matrix $\mat{cc}U&V\rix$ with
$U=\mat{ccc}u_1&\dots&u_\ell\rix$ and $V=\mat{ccc}v_1&\dots&v_s\rix$ are linearly independent.
Thus, let $X\in\mathbb C^{n\times (n-s-\ell)}$ be such that $\mat{ccc}X&U&V\rix$ is invertible and set
$P:=\mat{ccc}X&U&V\rix^{-1}$. Then we obtain
\[
PE(\lambda)P^*=\bmat{&\scriptstyle n-s-\ell & \scriptstyle \ell & \scriptstyle s\cr
\scriptstyle n-s-\ell & 0 & 0 & S_A^*+\lambda S_B^*\cr \scriptstyle\ell &0&D_A+\lambda D_B&\ast\cr
\scriptstyle s & S_A+\lambda S_B & \ast & \ast},
\]
where $S_A+\lambda S_B$ are the first $n-s-\ell$ columns of $(W_A^{ *}+\lambda W_B^{ *})P^*$, and where
$D_A,D_B,W_A,W_B$ are as in Remark~\ref{rem:concise}. In particular, all eigenvalues of $D_A+\lambda D_B$
are real and semisimple, because the pencil $D_A+\lambda D_B$ is diagonal. Furthermore, we can assume that
if $D_A+\lambda D_B$ has a multiple eigenvalue, say $\mu$, then all signs in the sign characteristic
of $D_A+\lambda D_B$ associated with $\mu$ are equal. Otherwise, we may use the trick from the part ``$\leq$''
to get a decomposition of the form~\eqref{herm-rank1} with an even smaller $\ell$.

Note that $S_A+\lambda S_B$ must be of full normal rank $s$, because otherwise the pencil $E(\lambda)$
would have less than $r=s+\ell+s$ linearly independent columns. Thus, in particular $S_A-\eta S_B$ has rank $s$
for all values $\eta\in\mathbb C$ that are not eigenvalues of $E(\lambda)$. This implies that the only eigenvalues
of $E(\la)$ are the eigenvalues of $D_A+\la D_B$. Moreover, if we denote the eigenvalues of $E(\la)$ by $\mu_1,\hdots,\mu_d$,
with respective algebraic multiplicities $m_1,\hdots,m_d$, then we have $\ell=\sum_{j=1}^d m_j$. Now, it suffices to prove
that $m_j=|\operatorname{sigsum}(E,\mu_j)|$, for $j=1,\hdots,d$. This will prove that $\ell=\ell_0$, a contradiction to
the assumption $\ell<\ell_0$. So let $\mu$ be one of
the eigenvalues of $D_A+\lambda D_B$, i.e., $\mu$ is real (or infinite). Suppose first that $\mu\in\mathbb R$.
Then for sufficiently small $\varepsilon>0$, we have that no $\widehat\lambda\in[\mu-\varepsilon,\mu+\varepsilon]\setminus\{\mu\}$
is an eigenvalue of $E(\lambda)$. Consequently, for all such $\widehat\lambda$, there exist a nonsingular matrix
$M\in \mathbb C^{(n-s-\ell)\times(n-s-\ell)}$ (depending on $\widehat\lambda$) such that
\[
(S_A+\widehat\lambda S_B)M=\bmat{ & \scriptstyle n-r & \scriptstyle s\cr \scriptstyle s & 0& S},
\]
where $S\in\mathbb C^{s\times s}$ is invertible (and also depends on $\widehat\lambda$). But this implies that
\[
\mat{ccc}M^*&0&0\\ 0&I&0\\ 0&0&I\rix PE(\widehat\lambda)P^*\mat{ccc}M&0&0\\ 0&I&0\\ 0&0&I\rix
=\bmat{ & \scriptstyle n-r & \scriptstyle s & \scriptstyle \ell & \scriptstyle s \cr
\scriptstyle n-r & 0&0&0&0\cr \scriptstyle s&0&0&0&S^*\cr \scriptstyle\ell &0&0&D_A+\widehat\lambda D_B&\ast\cr
\scriptstyle s &0&S&\ast& \ast},
\]
and due to the nonsingularity of $S$, we can easily read off the inertia index from the Hermitian
matrix $E(\widehat\lambda)$. If $\operatorname{ind}(H)=(\nu_+,\nu_-,\nu_0)$ denotes the inertia index of
a given Hermitian matrix $H$, i.e., $\nu_+$, $\nu_-$, and $\nu_0$ are the numbers of positive, negative, and
zero eigenvalues of $H$ (counted with multiplicities), respectively, then we easily obtain
(see also \cite[Lemma 6]{Meh00}) that
\[
\operatorname{ind}\big(E(\widehat\lambda)\big)=(s,s,n-r)+\operatorname{ind}(D_A+\widehat\lambda D_B),
\]
where the sum of triples is taken componentwise. Assume that $\operatorname{ind}(D_A+\mu D_B)=(d_+,d_-,m)$,
i.e., $m$ is the algebraic multiplicity of the eigenvalue $\mu$ of $D_A+\mu D_B$.
Then it follows that
\[
\operatorname{ind}\big(E(\mu-\varepsilon)\big)=(s+d_+,s+d_-+m,n-r)\quad\mbox{and}\quad
\operatorname{ind}\big(E(\mu+\varepsilon)\big)=(s+d_++m,s+d_-,n-r)
\]
if the sign of $\mu$ in the sign characteristic of $D_A+\lambda D_B$ is positive (recall that
all signs associated with $\mu$ in the sign characteristic of $D_A+\lambda D_B$ are equal), or
\[
\operatorname{ind}\big(E(\mu-\varepsilon)\big)=(s+d_++m,s+d_-,n-r)\quad\mbox{and}\quad
\operatorname{ind}\big(E(\mu+\varepsilon)\big)=(s+d_+,s+d_-+m,n-r)
\]
if the sign of $\mu$ in the sign characteristic of $D_A+\lambda D_B$ is negative.
Similarly, checking the change of inertia index of $E(\widehat\lambda)$ based on its Hermitian canonical form,
a straightforward computation shows that the number of positive or negative eigenvalues change by
the number $\operatorname{sigsum}(\mu)$ when $\widehat\lambda$ passes from $\mu-\varepsilon$ to
$\mu+\varepsilon$. This shows that we must have $m=|\operatorname{sigsum}(\mu)|$.

Finally, assume that $\mu=\infty$ is an eigenvalue of $D_A+\lambda D_B$ with algebraic multiplicity $m$.
If $\eta>0$ is sufficiently large such that all finite eigenvalues of $E(\lambda)$ are contained in the
interval $]-\eta,\eta[$, then a similar comparison of the inertia indices of $E(\eta)$ and $E(-\eta)$
reveals that the algebraic multiplicity of $\infty$ as an eigenvalue of $D_A+\lambda D_B$ must be
$|\operatorname{sigsum}(\infty)|$. 
\end{proof}

\subsection{Rank-$1$ decomposition for other structures}

Next, we consider a decomposition analogous to~\eqref{herm-rank1} for the other structures mentioned
at the beginning of this section. For most of these decompositions, observations similar to the ones
in Remark~\ref{rem:concise}--\ref{rem:linind} can be made, but for the
sake of brevity we refrain from stating them explicitly.
\begin{theorem}\label{rank1-sym.th} {\rm(Rank-$1$ decomposition for symmetric pencils).}
If $E(\la)$ is a symmetric $n\times n$ matrix pencil with $\rank E=r\leq n$, then it can be written as
\begin{equation}\label{sym-rank1}
E(\la)=(a_1+\la b_1)u_1u_1^\top+\cdots+(a_\ell+\la b_\ell)u_\ell u_\ell^\top+v_1w_1^\top+\cdots+v_sw_s^\top
+w_1v_1^\top+\cdots+w_sv_s^\top,
\end{equation}
where $a_i,b_i\in\CC$, for $i=1,\hdots,\ell$, and
\begin{itemize}
\item[(i)] $\ell+2s=r$,
\item[(ii)] $\deg u_1=\cdots=\deg u_\ell=0=\deg v_1=\cdots=\deg v_s$ and $\deg w_1,\hdots,\deg w_s\leq 1$.
\end{itemize}
\end{theorem}
\begin{proof}
The proof is similar to the one of Theorem \ref{rank1-herm.th} using the canonical form
for complex symmetric pencils \cite[Theorem 2.17]{batzke-thesis}. The only difference with the Hermitian case
is that in the symmetric case complex eigenvalues are not necessarily paired up by conjugation,
so terms of the form $(a+\la b)vv^\top$ may come also from odd blocks associated with complex eigenvalues.
\end{proof}
\begin{remark}\rm
The minimal value of $\ell$ is achieved when all eigenvalues of the pencil
\[
D_A+\lambda D_B:=\operatorname{diag}(a_1,\dots,a_\ell)+\lambda\operatorname{diag}(b_1,\dots,b_\ell),
\]
\cm{as in Remark~\ref{rem:concise}}, have algebraic multiplicity {\colb equal to $1$}. If the multiplicity is larger than {\colb $1$} for some eigenvalue
which is given, say, by the $i$th and $j$th diagonal entries $a+\lambda b$ and $c(a+\lambda b)$, with some
$c\in\mathbb C\setminus\{0\}$,
then with a similar trick as in the proof of Theorem~\ref{thm:sigsum} two summands of the form
$(a+\lambda b)u_iu_i^\top+(ca+\lambda cb)u_ju_j^\top$ can be replaced
by two summands of the form $v_kw_k^\top+w_kv_k^\top$ by choosing $v_k=\frac{1}{2}(u_i+\iunit du_j)$ and
$w_k=a(u_i-\iunit du_j)+\lambda b(u_i-\iunit du_j)$, where $d\in\mathbb C$ is a square root of $c$, i.e., $d^2=c$.
On the other hand, each eigenvalue of $E(\lambda)$ with odd algebraic multiplicity must occur
in one of the summands $(a+\lambda b)u_iu_i^\top$. Indeed, similar to Remark~\ref{rem:linind} we
can show that the vectors $u_1,\dots,u_\ell,v_1,\dots,v_s$ are linearly independent and with an
argument similar to the one in the proof of Theorem~\ref{thm:sigsum}, we can show that $E(\lambda)$
is congruent to a pencil of the form
\[
\bmat{&\scriptstyle n-s-\ell &\scriptstyle \ell & \scriptstyle s\cr
\scriptstyle n-s-\ell & 0 & 0 & S_A^\top+\lambda S_B^\top\cr \scriptstyle\ell &0&D_A+\lambda D_B&\ast\cr
\scriptstyle s & S_A+\lambda S_B & \ast & \ast},
\]
which shows that any eigenvalue that is not an eigenvalue of $D_A+\lambda D_B$ must have even
algebraic multiplicity being an eigenvalue of both $S_A+\lambda S_B$ and $S_A^\top+\lambda S_B^\top$.
Thus, we have just shown that the minimal value of $\ell$ is equal to the number of pairwise distinct
eigenvalues of $E(\lambda)$ that have odd algebraic multiplicity.
\end{remark}

We highlight in passing that in the case of complex symmetric matrices and other structures that are based on
the transpose rather than the Hermitian transpose no sign characteristic is involved.
\begin{theorem}\label{rank1-sksym.th} {\rm(Rank-$1$ decomposition for skew-symmetric pencils).}
If $E(\la)$ is a skew-symmetric $n\times n$ matrix pencil with $\rank E=r\leq n$, then $r$ is even and
$E(\la)$ can be written as
\begin{equation}\label{sksym-rank11}
E(\la)=v_1w_1^\top+\cdots+v_{s}w_{s}^\top-w_1v_1^\top-\cdots-w_{s}v_{s}^\top,
\end{equation}
where $s=\frac{r}{2}$, $\deg v_1=\cdots=\deg v_s=0$, and $\deg w_1,\hdots,\deg w_s\leq 1$.
\end{theorem}
\begin{proof}
The proof follows the same steps as the proof of Theorem \ref{rank1-herm.th}. All blocks in the
skew-symmetric canonical form are paired up (see \cite[Theorem 2.18]{batzke-thesis}). More precisely,
the blocks in this canonical form are of three different kinds, namely:
(a) pairs of $k\times k$ blocks associated with the eigenvalue $\infty$,
(b) pairs of $k\times k$ blocks associated with a complex eigenvalue,
and (c) pairs of a $k\times (k+1)$ right singular and a $(k+1)\times k$ left singular block.
Then, following the proof of Theorem \ref{rank1-herm.th}, we can decompose any of these blocks
as a sum of rank-$1$ pencils as in \eqref{sksym-rank11}.
\end{proof}
\begin{theorem}\label{rank1-teven.th} {\rm(Rank-$1$ decomposition for $\top$-even pencils).}
If $E(\la)$ is a $\top$-even $n\times n$ matrix pencil with $\rank E=r\leq n$, then it can be written as
\begin{equation}\label{teven-rank1}
E(\la)=\left\{\begin{array}{ll}
v_1w_1(\la)^\top+\cdots+v_{s}w_{s}(\la)^\top+ w_1(-\la)v_1^\top+\cdots
+ w_{s}(-\la)v_{s}^\top,&\mbox{if $r$ is even,}\\
uu^\top+v_1w_1(\la)^\top+\cdots+v_{s}w_{s}(\la)^\top+ w_1(-\la)v_1^\top+\cdots
+ w_{s}(-\la)v_{s}^\top,&\mbox{if $r$ is odd,}
\end{array}\right.
\end{equation}
where $s={\lfloor r/2\rfloor}$, $\deg u=\deg v_1=\cdots=\deg v_s=0$ and $\deg w_1,\hdots,\deg w_s\leq 1$.
\end{theorem}
\begin{proof} We proceed in a similar way as in the proof of Theorem \ref{rank1-herm.th} using the
canonical form for $\top$-even pencils \cite[Theorem 2.16]{batzke-thesis}. Again, we may assume the
$\top$-even pencil $L(\la)$ is given in canonical form. Then, it is a direct sum of blocks of six kinds,
namely: (a) $(2k+1)\times(2k+1)$ blocks associated with the eigenvalue $\infty$,
(b) pairs of $(2\ell)\times(2\ell)$ blocks associated with the eigenvalue $\infty$,
(c) pairs of $(2m+1)\times(2m+1)$ blocks associated with the eigenvalue $0$, (d) $(2p)\times(2p)$ blocks
associated with the eigenvalue $0$, (e) pairs of $q\times q$ blocks corresponding to a pair of eigenvalues
$\mu,-\mu\in\CC\setminus\{0\}$, and (f) pairs of a right and a left singular block of size $(r+1)\times r$
and $r\times (r+1)$, respectively. Blocks of type (d) can be written as a sum of two rank-$1$ pencils of
the form $vw^\top+wv^\top$ using the same decomposition as in the proof of Theorem \ref{rank1-herm.th}.
Similarly, paired blocks of types (b)--(c) and (e)--(f) can be written as a sum of paired rank-$1$
pencils $vw^\top+wv^\top$ using a combined row-column expansion. For instance, a pair of blocks of type (e)
has the form
\[
\left[\begin{array}{cccc|cccc}
&&&&&&&\mu+\la\\
&&&&&&\iddots&1\\
&&&&&\mu+\la&\iddots&\\
&&&&\mu+\la&1&&\\\hline
&&&\mu-\la&&&&\\
&&\mu-\la&1&&&&\\
&\iddots&\iddots&&&&&\\
\mu-\la&1&&&&&&
\end{array}\right]_{(2q)\times(2 q)}
\]
and can be decomposed into a sum $v_1w_1(\la)^\top+\cdots+v_qw_q(\la)^\top+w_1(-\la)v_1^\top+
\cdots+w_q(-\la)v_q^\top$ of $2q$ rank-$1$ pencils
with $v_i=e_{2q-i+1}$, for $i=1,\hdots,q$, and $w_i(\la)$ being, up to the sign, the $(2q-i+1)$th column of the whole matrix pencil,
namely $w_i(\la)=\left[\begin{array}{cccc}0_{i-1}&\mu{\colb-}\la&1&0_{2q-i-1}\end{array}\right]^\top$
for $i=1,\hdots,q-1$, and $w_q(\la)=\left[\begin{array}{cccc}0_{q-1}&\mu{\colb-}\la&0_{q}\end{array}\right]^\top$.
Blocks of type (a), however, will need one extra term of the form $uu^\top$. To be more precise,
the $(2k+1)\times(2k+1)$ block associated with $\infty$ having the form
\[
\left[
\begin{array}{ccccccc}
&&&&&&1\\&&&&&\iddots&\la\\
&&&&1&\iddots&\\
&&&1&\la&&\\
&&\iddots&-\la&&\\
&1&\iddots&&&&\\
1&-\la&&&&&
\end{array}
\right]_{(2k+1)\times(2k+1)}
\]
can be decomposed as
$uu^\top+v_1w_1(\la)^\top+\cdots+v_{k}w_k(\la)^\top+w_1(-\la)v_1^\top+\cdots+w_k(-\la)v_k^\top$,
where $u=e_k,v_i=e_{2k-i+2}$ for $i=1,\hdots,k$, and where \vm{for $i=1,\hdots,k$, $w_i(\la)^\top$
is the $(2k-i+2)$th} row of the matrix pencil, namely
$w_i(\la)=\left[\begin{array}{cccc}0_{1\times(i-1)}&1&-\la&0_{1\times(2k-i)}\end{array}\right]^\top$.

The previous arguments show that $E(\la)$ can be written as
\begin{equation}\label{teven-step}
\begin{array}{ccl}
E(\la)&=&u_1u_1^\top+\cdots+u_\ell u_\ell^\top+v_1w_1(\la)^\top+\cdots+v_{s}w_s(\la)^\top\\
&&+w_1(-\la)v_1^\top+\cdots+w_s(-\la)v_s^\top,
\end{array}
\end{equation}
with $\ell+2s=r$, and $\deg u_1=\hdots=\deg u_\ell=\deg v_1=\hdots=\deg v_s=0$.
It remains to prove that, given two vectors $u_1,u_2\in\CC^n$, there exist another two vectors $v,w$,
with $\deg v=0$, such that
\begin{equation}\label{rank1-rank2}
u_1u_1^\top+u_2u_2^\top=vw^\top+wv^\top.
\end{equation}
Note that, if this is true, then we can group an even number of summands of the form $uu^\top$ in
\eqref{teven-step} to get a decomposition like in \eqref{teven-rank1}.

To get the expression \eqref{rank1-rank2}, just set $v=u_1+\iunit u_2$ and $w=\frac{1}{2}(u_1-\iunit u_2)$.
\end{proof}
\begin{theorem}\label{rank1-todd.th} {\rm(Rank-$1$ decomposition for $\top$-odd pencils).}
If $E(\la)$ is a $\top$-odd $n\times n$ matrix pencil with $\rank E=r\leq n$, then it can be written as
\begin{equation}\label{todd-rank1}
E(\la)=\left\{\begin{array}{ll}
v_1w_1(\la)^\top+\cdots+v_{s}w_{s}(\la)^\top- w_1(-\la)v_1^\top-\cdots
- w_{s}(-\la)v_{s}^\top,&\mbox{if $r$ is even,}\\
\lambda uu^\top+v_1w_1(\la)^\top+\cdots+v_{s}w_{s}(\la)^\top- w_1(-\la)v_1^\top-\cdots
- w_{s}(-\la)v_{s}^\top,&\mbox{if $r$ is odd,}
\end{array}\right.
\end{equation}
where $s={\lfloor r/2\rfloor}$, $\deg u=\deg v_1=\cdots=\deg v_s=0$ and $\deg w_1,\hdots,\deg w_s\leq 1$.
\end{theorem}
\begin{proof}
The result follows from Theorem~\ref{rank1-teven.th} applied to the reversal of $E(\la)$ and using
Remark~\ref{reversal.rem}.
\end{proof}

The following decomposition for low-rank $\top$-palindromic pencils has been presented in the recent
reference \cite[Th. 3.1]{d18}. For completeness, we provide a different proof based on Theorem~\ref{rank1-herm.th}.
\begin{theorem}\label{rank1-tpal.th} {\rm(Rank-$1$ decomposition for $\top$-palindromic pencils).}
If $E(\la)$ is a $\top$-palindromic $n\times n$ matrix pencil with $\rank E=r\leq n$, then it can be written as
\begin{equation}\label{tpal-rank1}
E(\la)=\left\{\begin{array}{lc}
v_1w_1^\top+\cdots+v_{s}w_{s}^\top+ (\rev w_1)v_1^\top+\cdots+ (\rev w_{s})v_{s}^\top,&\mbox{if $r$ is even,}\\
(1+\la)uu^\top+v_1w_1^\top+\cdots+v_{s}w_{s}^\top+(\rev w_1)v_1^\top+\cdots+
(\rev w_{s})v_{s}^\top,&\mbox{if $r$ is odd,}
\end{array}\right.
\end{equation}
where $s={\lfloor r/2\rfloor}$, $\deg u=\deg v_1=\cdots=\deg v_s=0$ and $\deg w_1,\hdots,\deg w_s\leq 1$.
\end{theorem}
\begin{proof} The result follows from Theorem~\ref{rank1-herm.th} using Cayley transformations. More precisely,
let ${\cal C}_{-1}$ and ${\cal C}_{+1}$ be the Cayley transformations of a given matrix pencil $P(\la)$
defined as
\begin{equation}\label{cayley}
{\cal C}_{-1}(P)(\la)=(1+\la)P\left(\frac{\la-1}{1+\la}\right)\quad\mbox{and}\quad
{\cal C}_{+1}(P)(\la)=(1-\la)P\left(\frac{1+\la}{1-\la}\right).
\end{equation}
It is known that, if $E(\la)$ is $\top$-palindromic, then ${\cal C}_{+1}(E)$ is $\top$-even
\cite[Theorem 2.7]{4m-good}. It is clear, by definition, that both ${\cal C}_{-1}$ and ${\cal C}_{+1}$
preserve the rank. Then ${\cal C}_{+1}(E)$ is $\top$-even with $\rank {\cal C}_{+1}(E)=r$,
so it admits a decomposition like \eqref{teven-rank1}. We will focus on the case when $r$ is odd, because
the case when $r$ is even is analogous. Using that
${\cal C}_{-1}({\cal C}_{+1}(P))(\la)=2P(\la)$ for any matrix pencil $P(\la)$,
see \cite[Proposition 2.5]{4m-good}, it follows that
\[
\begin{array}{ccl}
2E(\la)&=&{\cal C}_{-1}\left(uu^\top+\sum_{j=1}^s \big(v_jw_j(\la)^\top+w_j(-\la)v_j^\top\big)\right)\\
&=&(1+\la) uu^\top+\sum_{j=1}^sv_j\left((1+\la)w_j\left(\frac{\la-1}{1+\la}\right)\right)^\top
+\sum_{j=1}^s\left((1+\la)w_j\left(\frac{1-\la}{1+\la}\right)\right)v_j^\top,
\end{array}
\]
where $s=(r-1)/2$. Now, the result follows from the identity
\begin{equation}\label{reversal-w}
\rev\left((1+\la)w\left(\frac{\la-1}{1+\la}\right)\right)=\la\left(1+\frac{1}{\la}\right)w
\left(\frac{\frac{1}{\la}-1}{1+\frac{1}{\la}}\right)=(1+\la)w\left(\frac{1-\la}{1+\la}\right).
\end{equation}
%
\end{proof}

Using again appropriate Cayley transformations and the decomposition for $\top$-even matrix pencils in
Theorem~\ref{rank1-teven.th} we can also get a rank-$1$ decomposition for $\star$-anti-palindromic pencils.
\begin{theorem}\label{rank1-tantipal.th} {\rm(Rank-$1$ decomposition for $\top$-anti-palindromic pencils).}
If $E(\la)$ is a $\top$-anti-palindromic $n\times n$ matrix pencil with $\rank E=r\leq n$, then it can be
written as
\begin{equation}\label{tantipal-rank1}
\begin{array}{cc}
E(\la)=&\left\{\begin{array}{cl}
v_1w_1^\top+\cdots+v_{s}w_{s}^\top- (\rev w_1)v_1^\top-\cdots- (\rev w_{s})v_{s}^\top,&\mbox{if $r$ is even,}\\
(1-\la)uu^\top+v_1w_1^\top+\cdots+v_{s}w_{s}^\top-(\rev w_1)v_1^\top-\cdots- (\rev w_{s})v_{s}^\top,&\mbox{if $r$ is odd,}
\end{array}\right.
\end{array}
\end{equation}
where $s=\lfloor r/2\rfloor=0$, $\deg v_1=\cdots=\deg v_s=0$, and $\deg w_1,\hdots,\deg w_s\leq 1$.
\end{theorem}
\begin{proof}
The proof is similar to the one of Theorem \ref{rank1-tpal.th}, but first considering ${\cal C}_{-1}(E)$,
which is $\top$-even \cite[Th. 2.7]{4m-good}, and then applying ${\cal C}_{+1}$ to get
${\cal C}_{+1}({\cal C}_{-1}(E))=2E$. The differences between~\eqref{tantipal-rank1} and~\eqref{tpal-rank1}
come from the identities
\[
\begin{array}{ll}
{\cal C}_{+1}(uu^\top)=(1-\la)uu^\top,\\
{\cal C}_{+1}(vw(\la)^\top)=v\left((1-\la)w\left(\frac{1+\la}{1-\la}\right)\right)^\top,&
{\cal C}_{+1}(w(-\la)v^\top)=(1-\la)w\left(\frac{1+\la}{\la-1}\right)v^\top,
\end{array}
\]
and
\[
\rev\left((1-\la)w\left(\frac{1+\la}{1-\la}\right)\right)=\la\left(1-\frac{1}{\la}\right)w
\left(\frac{1+\frac{1}{\la}}{1-\frac{1}{\la}}\right)=-(1-\la)w\left(\frac{1+\la}{\la-1}\right).
\]
\end{proof}

We highlight that the parameter $\ell$ in the decomposition $r=\ell+2s$ takes the minimal value
zero or one in the decompositions in Theorem~\ref{rank1-sksym.th}--\ref{rank1-tantipal.th}. This is
in contrast with Theorem~\ref{rank1-herm.th} and Theorem~\ref{rank1-sym.th}, where the minimal
value for $\ell$ can be as large as $r$, for example if the pencil $E(\lambda)$ does only have
simple eigenvalues in the symmetric case, or only simple real eigenvalues in the Hermitian case.

The rank-$1$ decompositions for skew-Hermitian, $*$-even, and $*$-odd pencils can be directly obtained
from the decomposition in the Hermitian case, by means of the following observation
(see \cite[page 80]{batzke-thesis}):
\begin{itemize}
\item If $A+\la B$ is skew-Hermitian then $\iunit (A+\la B)$ is Hermitian.
\item If $A+\la B$ is $*$-even then $A+\la (\iunit B)$ is Hermitian.
\item $A+\la B$ is $*$-odd if and only if $B+\la A$ is $*$-even.
\end{itemize}
For completeness, we explicitly state these decompositions in a similar way as we have done for
the previous structures.
\begin{theorem}\label{rank1-sk.th} {\rm(Rank-$1$ decomposition for skew-Hermitian pencils).}
If $E(\la)$ is a skew-Hermitian $n\times n$ matrix pencil with $\rank E=r\leq n$, then it can be written as
\begin{equation}\label{sk-rank1}
E(\la)=\iunit(a_1+\la b_1)u_1u_1^*+\cdots+\iunit(a_\ell+\la b_\ell)u_\ell u_\ell^*+v_1w_1^*+\cdots
+v_sw_s^*-w_1v_1^*-\cdots-w_sv_s^*,
\end{equation}
where $a_i,b_i\in\RR$, for $i=1,\hdots,\ell$, and
\begin{itemize}
\item[(i)] $\ell+2s=r$,
\item[(ii)] $\deg u_1=\cdots=\deg u_\ell=0=\deg v_1=\cdots=\deg v_s$ and $\deg w_1,\hdots,\deg w_s\leq 1$.
\end{itemize}
\end{theorem}
\begin{theorem}\label{rank1-stareven.th} {\rm(Rank-$1$ decomposition for $*$-even pencils).}
If $E(\la)$ is a $*$-even $n\times n$ matrix pencil with $\rank E=r\leq n$, then it can be written as
\begin{equation}\label{*even-rank1}
\begin{array}{cl}
E(\la)=&(a_1+\la (b_1\iunit))u_1u_1^*+\cdots+(a_\ell+\la (b_\ell\iunit))u_\ell u_\ell^*\\
&+v_1w_1(\la)^*+\cdots+v_sw_s(\la)^*+w_1(-\la)v_1^*+\cdots+w_s(-\la)v_s^*,
\end{array}
\end{equation}
where $a_i,b_i\in\RR$, for $i=1,\hdots,\ell$, and
\begin{itemize}
\item[(i)] $\ell+2s=r$,
\item[(ii)] $\deg u_1=\cdots=\deg u_\ell=0=\deg v_1=\cdots=\deg v_s$ and $\deg w_1,\hdots,\deg w_s\leq 1$.
\end{itemize}
\end{theorem}

\begin{theorem}\label{rank1-starodd.th} {\rm(Rank-$1$ decomposition for $*$-odd pencils).}
If $E(\la)$ is a $*$-odd $n\times n$ matrix pencil with $\rank E=r\leq n$, then it can be written as
\begin{equation}\label{*odd-rank1}
\begin{array}{cl}
E(\la)=&(a_1\iunit+\la b_1)u_1u_1^*+\cdots+(a_\ell\iunit+\la b_\ell)u_\ell u_\ell^*\\
&+v_1w_1(\la)^*+\cdots+v_sw_s(\la)^*-w_1(-\la)v_1^*-\cdots-w_s(-\la)v_s^*,
\end{array}
\end{equation}
where $a_i,b_i\in\RR$, for $i=1,\hdots,\ell$, and
\begin{itemize}
\item[(i)] $\ell+2s=r$,
\item[(ii)] $\deg u_1=\cdots=\deg u_\ell=0=\deg v_1=\cdots=\deg v_s$ and $\deg w_1,\hdots,\deg w_s\leq 1$.
\end{itemize}
\end{theorem}

The decomposition in \eqref{sk-rank1} follows from \eqref{herm-rank1} after multiplying by $\iunit$ and
using that, for any pair of vectors $u,v\in\CC[\la]^n$, we can write
$\iunit(uw^*+wv^*)=(\iunit v)w^*-w(\iunit v)^*=\widetilde vw^*-w\widetilde v^*$, with $\widetilde v=\iunit v$.
Similarly, the expression \eqref{*even-rank1} follows from \eqref{herm-rank1} applied to $E(\iunit\la)$ and then
multiplying \vm{ the leading coefficient in the decomposition by $-\iunit$}. Note that, if
$A+\la (\iunit B)=vw(\la)^*+w(\la)v^*=v(w_0^*+\la w_1^*)+(w_0+\la w_1)v^*$ (with $v\in\CC^n$ and $w(\la)=w_0+\la w_1$,
$w_0,w_1\in\CC^n$), then, \vm{multiplying the leading coefficient by $-\iunit$,} we get
$
A+\la B=v(w_0^*-\iunit \la w_1^*)+(w_0-\iunit \la w_1)v^*=v(w_0^*+\la (\iunit w_1)^*)+(w_0-\la (\iunit w_1))v^*=vw(\la)^*+w(-\la)v^*
$. Finally, the decomposition \eqref{*odd-rank1} follows from
\eqref{*even-rank1} applied to $\rev E(\la)$ and then applying the reversal to the decomposition in the right-hand side.
Note that, if $\la A+B=v\widetilde w(\la)^*+\widetilde w(-\la)v^*=v(w_0^*+\la w_1^*)+(w_0-\la w_1)v^*$ (with $v\in\CC^n$
and $\widetilde w(\la)=w_0+\la w_1$, $w_0,w_1\in\CC^n$), then $A+\la B=v(w_1^*+\la w_0^*)-(w_1-\la w_0)v^*=vw(\la)^*-w(-\la)v^*$,
where $w(\la)=\rev\widetilde w(\la)$.

As for the $*$-palindromic structure, the decomposition follows from \eqref{*even-rank1} using appropriate
Cayley transforms, like for the $\top$-palindromic structure.
 \begin{theorem}\label{rank1-starpal.th} {\rm(Rank-$1$ decomposition for $*$-palindromic pencils).}
If $E$ is a $*$-palindromic $n\times n$ matrix pencil with $\rank E=r\leq n$, then it can be written as
\begin{equation}\label{*pal-rank1}
\begin{array}{cl}
E(\la)=&\big((a_1-b_1\iunit)+\la(a_1+b_1\iunit)\big)u_1u_1^*+\cdots+\big((a_\ell-b_\ell\iunit)+
\la(a_\ell+b_\ell\iunit)\big)u_\ell u_\ell^* \\
&+v_1w_1^*+\cdots+v_sw_s^*+(\rev w_1)v_1^*+\cdots+(\rev w_s)v_s^*,
\end{array}
\end{equation}
where $a_i,b_i\in\RR$, for $i=1,\hdots,\ell$, and
\begin{itemize}
\item[(i)] $\ell+2s=r$,
\item[(ii)] $\deg u_1=\cdots=\deg u_\ell=0=\deg v_1=\cdots=\deg v_s$ and $\deg w_1,\hdots,\deg w_s\leq 1$.
\end{itemize}
\end{theorem}
\begin{proof}
The proof is similar to the one of Theorem \ref{rank1-tpal.th}, but we include it here to illustrate where
the difference in the first $\ell$ summands comes from. In particular, if $E(\la)$ is $*$-palindromic as in
the statement, then ${\cal C}_{+1}(E)$ is $*$-even \cite[Theorem 2.7]{4m-good}. Therefore, it admits a
decomposition like \eqref{*even-rank1}. Now
\[
\begin{array}{crl}
2E(\la)=&{\cal C}_{-1}\big({\cal C}_{+1}(E)\big)=&{\cal C}_{-1}\left(\sum_{i=1}^\ell
\big(a_i+\la (b_i\iunit)\big)u_iu_i^*\right)
+{\cal C}_{-1}\left(\sum_{j=1}^s\big(v_jw_j(\la)^*+w_j(-\la)v_j^*\big)\right)\\
&=&\sum_{i=1}^\ell\big((a_i-b_i\iunit)+\la(a_i+b_i\iunit)\big)u_iu_i^*+\sum_{j=1}^s\big(v_jw_j^*+(\rev w_j)v_j^*\big),
\end{array}
\]
where, for the first sum, we have used that
\[
{\cal C}_{-1}\big((a+\la (b\iunit))uu^*\big)=(1+\la)\left(a+\frac{\la-1}{1+\la}\,b\iunit\right)uu^*=
\big((a-b\iunit)+\la(a+b\iunit)\big)uu^*,
\]
and, for the second sum, we have followed exactly the same steps as in the proof of Theorem \ref{rank1-tpal.th},
just replacing $\top$ by $*$.
\end{proof}
%
Note that the first $\ell$ summands in the right-hand side of \eqref{*pal-rank1} come from eigenvalues
of $E(\la)$ which lie on the unit circle. Moreover, any complex value on the unit circle can be identified as
a root of a linear polynomial of the form $(a-b\iunit)+\la(a+b\iunit)$.
%
\begin{theorem}\label{rank1-starantipal.th} {\rm(Rank-$1$ decomposition for $*$-anti-palindromic pencils).}
If $E(\la)$ is a $*$-anti--palindromic $n\times n$ matrix pencil with $\rank E=r\leq n$,
 then it can be written as
\begin{equation}\label{*antipal-rank1}
\begin{array}{cl}
E(\la)=&((a_1+b_1\iunit)+\la(-a_1+b_1\iunit))u_1u_1^*+\cdots+((a_\ell+b_\ell\iunit)+\la(-a_\ell+b_\ell\iunit))u_\ell u_\ell^* \\
&+v_1w_1^*+\cdots+v_sw_s^*-(\rev w_1)v_1^*-\cdots-(\rev w_s)v_s^*,
\end{array}
\end{equation}
where $a_i,b_i\in\RR$, for $i=1,\hdots,\ell$, and
\begin{itemize}
\item[(i)] $\ell+2s=r$,
\item[(ii)] $\deg u_1=\hdots=\deg u_\ell=0=\deg v_1=\hdots=\deg v_s$ and $\deg w_1,\hdots,\deg w_s\leq 1$.
\end{itemize}
\end{theorem}
\begin{proof}
The proof follows the same steps as the proof of Theorem \ref{rank1-tantipal.th}.
\end{proof}
Concerning minimality of the parameter $\ell$, there is a characterization analogous to the one in
Theorem~\ref{thm:sigsum} involving the signsum of real eigenvalues in the case of skew-Hermitian pencils,
of purely imaginary eigenvalues in the case of $*$-even and $*$-odd pencils, or unimodular eigenvalues
in the case of $*$-palindromic or $*$-anti-palindromic pencils. We refrain from explicitly stating
these characterizations.

\section{Structure-preserving low-rank perturbations}\label{main.sec}

In this section, we will develop our main results on the change of the partial multiplicities
of eigenvalues of matrix pencils with symmetry structure under generic structure-preserving
low-rank perturbations. For this, we follow the approach in \cite{dd16}. More precisely, let $\sym_r$ be the set of
matrix pencils with structure $\sym$ and with rank at most $r$, where $\sym$ is any of the structures
mentioned in Section \ref{rank1.sec}, let $L(\la)$ be a regular pencil (with structure $\sym$)
and let $\la_0$ be an eigenvalue of $L(\la)$ (finite or infinite). The procedure then
consists of two main steps:
\begin{itemize}
\item[] {\bf Step 1.} Obtain a (polynomial) parameterization of $\sym_r$.
\item[] {\bf Step 2.} Prove that, for a generic set of parameters, all pencils $E(\la)\in\sym_r$
obtained from the previous parameterization are such that the partial multiplicities of $(L+E)(\la)$
at $\la_0$ are the ones described in the main results (given in Section \ref{pert.sec}).
\end{itemize}
{Step 1} is addressed in Section \ref{param.sec}, and {Step 2} is addressed in Section \ref{pert.sec}.
For the realization of {Step 2} we will need as a key ingredient a localization result that we develop in
Section~\ref{local.sec}, where we will also clarify the notion of genericity.

\subsection{A localization result}\label{local.sec}

Let $\mathbb F$ denote one of the fields $\mathbb R$ or $\mathbb C$, we then use the following notion of genericity.
\begin{definition}\label{gener.def}
A {\em generic set} $\cG$ of $\FF^m$ is a subset of $\FF^m$ whose complement is contained
in a proper algebraic set, i.e., $\cG$ is nonempty and coincides with the complement of a set of common zeros of
finitely many polynomials in $m$ variables.
\end{definition}

We highlight that even though in this paper we only deal with the case of complex matrix pencils, we have to use the
concept of genericity with respect to the real numbers when symmetry structures involving the conjugate transpose are
considered, because complex conjugation is not a polynomial map on $\mathbb C$. This problem can be circumvented if we
identify $\mathbb C^m$ with $\mathbb R^{2m}$ by considering the real and imaginary parts of each component separately.
In this context, complex conjugation is an $\mathbb R$-linear map and thus in particular polynomial.

We will need the following result, which is almost identical to \cite[Lemma 3.1]{MehR17}.
(The parameter $\mu$ will be equal to $\colb 1$ for most cases, which corresponds to simple eigenvalues.
However, in the case of skew-symmetric matrix pencils, considered in Theorem \ref{skews-main.th}, we will apply
the result with $\mu=2$.)
\begin{lemma}\label{simp}
Let $A\in\mathbb C^{n\times n}$ have the pairwise distinct eigenvalues
$\lambda_1,\dots,\lambda_\kappa\in\mathbb C$ with algebraic multiplicities $a_1,\dots,a_\kappa$, and
let $\varepsilon>0$ be such that the discs
\[
D_j:=\big\{\mu\in\mathbb C \,:\,|\lambda_j-\mu|<\varepsilon^{2/n}\big\},\quad j=1,\dots,\kappa
\]
are pairwise disjoint. Furthermore, let $U\subseteq\FF^m$ be open and
let $C:U\to\mathbb C^{n\times n}$ be an analytic function with $C(0)=A$, such that the following
conditions are satisfied:
\begin{enumerate}
\item[$1)$] For all $x\in U$, the algebraic multiplicity of any eigenvalue of
$C(u)$ is always a multiple of $\mu\in\mathbb N\setminus\{0\}$.
\item[$2)$] There exists a generic set $\cG\subseteq\FF^m$ such that, for all $x\in\cG\cap U$,
the matrix $C(x)$ has the eigenvalues $\lambda_1,\dots,\lambda_\kappa$
with algebraic multiplicities $\widetilde a_1,\dots,\widetilde a_\kappa$, where
$\widetilde a_j\leq a_j$ for $j=1,\dots,\kappa$. (Here, we allow $a_j=0$ in the case that $\lambda_j$ is no
longer an eigenvalue of $C(x)$.)
\item[$3)$] For each $j=1,\dots,\kappa$ there exists $x_j\in U$ with $\|x_j\|<\varepsilon$ such that
the matrix $C( x_j)$ has exactly $(a_j-\widetilde a_j)/\mu$ pairwise distinct eigenvalues in $D_j$
different from $\lambda_j$ and each one has algebraic multiplicity exactly $\mu$.
\end{enumerate}
Then there exists $\varepsilon'>0$ and a set $\cG_0$, open and dense in $\{x\in\FF^m\mid\|x\|<\varepsilon'\}$,
with $\cG_0\subseteq U$, such that, for all $x\in\cG_0$, the pencil
$C(x)$ has exactly $\sum_{j=1}^{\kappa}\frac{1}{\mu}(a_j-\widetilde a_j)$ eigenvalues
that are different from those of $A$ and each of these eigenvalues has
algebraic multiplicity exactly $\mu$.
\end{lemma}
\proof
The proof is almost identical to the one of Lemma 3.1 in \cite{MehR17} and therefore omitted.
(One just has to replace $\mathbb R$ in \cite{MehR17} with $\FF$ and remove the final
paragraph on the proof which is not needed here, because the statement of Lemma~\ref{simp}
has been adapted correspondingly.)
\hfill$\Box$

\medskip
The next result generalizes \cite[Theorem 3.2]{MehR17} (which itself was an extension of
\cite[Theorem 2.6]{BatMRR16}) from the matrix to the pencil case and will be the main tool
in Section~\ref{pert.sec}.
\begin{theorem}\label{local}
Let $L(\lambda)=A+\lambda B$ be a regular complex $n\times n$ matrix pencil and let $\lambda_1,...,\lambda_\kappa$
be its pairwise distinct eigenvalues (finite or infinite) with geometric multiplicities $g_i$, nonzero partial multiplicities
$n_{i,1}\geq n_{i,2}\geq\cdots\geq n_{i,g_i}>0$, and algebraic multiplicities
\[
a_i=\sum_{j=1}^{g_i}n_{i,j},
\]
for $i=1,\dots,\kappa$, respectively. Let $\Phi:\FF^m\rightarrow\CC^{n\times n}\times\CC^{n\times n}$ be a polynomial map and,
for $x\in\FF^m$, let us identify $\Phi(x)=(\Phi_A(x),\Phi_B(x))$ with the pencil $\Phi_A(x)+\la \Phi_B(x)$. Furthermore,
assume that, for all $x\in\FF^m$, we have
\begin{enumerate}
\item[{\rm (i)}] $\Phi(0)=(0,0)$;
\item[{\rm (ii)}] $\operatorname{rank}\Phi(x)\leq r$;
\item[{\rm (iii)}] if $L+\Phi(x)$ is regular, then the algebraic multiplicity of any eigenvalue of $L+\Phi(x)$
is always a multiple of some $\mu\in\mathbb N\setminus\{0\}$.
\end{enumerate}
Then the following statements hold:
\begin{enumerate}
\item[\rm (1)] If $x\in \FF^m$ is such that $ L+\Phi(x)$ is regular and if
$\eta_{i,1} \geq \cdots \geq \eta_{\widetilde g_i}$ are the partial multiplicities associated with $\lambda_i$ as an
eigenvalue of $ L+\Phi(x)$, for $i=1,\dots,\kappa$
(here we allow $\widetilde g_i=0$ if $\lambda_i$ is not an eigenvalue of $ L+\Phi(x)$), then the
list $(\eta_{i,1},\dots,\eta_{i,\widetilde g_i})$ dominates the list $(n_{i,r+1},\dots,n_{i,g_i})$,
i.e., we have $\widetilde g_i\geq g_i-r$ and $\eta_{i,j} \geq n_{i,j+r}$, for $j=1, \ldots , g_i-r$
and $i=1,\dots,\kappa$.
\item[\rm (2)]
Assume that, for all $x\in\FF^m$ for which $ L+\Phi(x)$ is regular, we have that, for each
$i=1,\dots,\kappa$, the algebraic multiplicity $a_i^{(x)}$ of $\lambda_i$ as an eigenvalue of $ L+\Phi(x)$
satisfies $a_i^{(x)}\geq \widetilde a_i$, for some $\widetilde a_i\in\mathbb{N}$.
If, for any  $\varepsilon>0$ and each $i=1,\dots,\kappa$, there exists $x_{0,i}\in\FF^m$ with $\|x_{0,i}\|<\varepsilon$
such that $ L+\Phi(x_{0,i})$ is regular,
$a_i^{(x_{0,i})} =\widetilde a_i$, and all eigenvalues of $ L+\Phi(x_{0,i})$ that are different from
those of $L$ have multiplicity precisely $\mu$,
then there exists a generic set $\cG\subseteq\FF^m$ such that, for all $x\in\cG$, the following conditions are satisfied:
\begin{itemize}
\item[\rm (a)] the pencil $ L+\Phi(x)$ is regular;
\item[\rm (b)] $a_i^{(x)} =\widetilde a_i$ for all $i=1,\dots,\kappa$;
\item[\rm (c)] all eigenvalues of $ L+\Phi(x)$ which are different from those of $L$ have multiplicity precisely $\mu$.
\end{itemize}
If, in addition, we have $\widetilde a_i=n_{i,r+1}+\cdots+n_{i,g_i}$ for some $i\in\{1,\dots,\kappa\}$,
then the partial multiplicities of $\lambda_i$ as an eigenvalue of $ L+\Phi(x)$ are precisely
$n_{i,r+1},\dots,n_{i,g_i}$ for all $x\in\cG$.
\end{enumerate}
\end{theorem}
\proof In order to introduce the dependence on $\la$ in the pencil $\Phi(x)$, we denote $\Phi_x(\la):=\Phi(x)$ along the proof.
First of all, we may assume that $\infty$ is not an eigenvalue of $L(\lambda)$. Otherwise,
consider instead the pencil $\widehat L(\lambda)=A+\lambda\big(\alpha A+B\big)$, for some
$\alpha\in\,]0,1[$ such that $\infty$ is not an eigenvalue of $\widehat L(\lambda)$. Note that this
transformation only changes the eigenvalues, but not their corresponding multiplicities and their
behavior under perturbation when the perturbation pencil is adapted to
$\widehat \Phi_x(\la)=\Phi_A(x)+\lambda\big(\alpha \Phi_A(x)+ \Phi_B(x)\big)$.

Part (1) is a direct consequence of \cite[Lemma 2.1]{ddm} using the fact that the
rank of $ \Phi_x(\la)$ is at most $r$, for any $x\in\mathbb \FF^m$.

For part (2), we first show that the set
\[\cG_{\rm reg} =\{ x\in\mathbb \FF^m\mid (L+ \Phi_x)(\la)\mbox{ is regular}\}\] is a generic set.
To see this, let $z\in\mathbb C$ be a value which is not an eigenvalue of $L(\lambda)$.
Then $p(x):=\det\big((L+\Phi_x)(z)\big)$ is a polynomial in the entries of
$x$ that is not the zero polynomial. The set of pencils for which $ L+\Phi_{x}$ is singular
is then contained in the set of pencils for which $p(x)=0$, which by definition is an algebraic set.
Therefore, \cm{$\cG_{\rm reg}$} is generic.

Next, let $Y_i(x)$ be the matrix $Y_i(x)=\big( (L+\Phi_x)(\lambda_i)\big)^n$.
Then, by assumption, we have ${\rm rank\, } Y_i(x_{0,i})=n-\widetilde a_i$, for some $x_{0,i}\in\FF^m$, and
it follows from \cite[Lemma 2.1]{MehMRR11} that the set
\[
\cG_{i}:=\{x\in\FF^m\mid {\rm rank\, }Y_i(x)\geq n-\widetilde a_i\}
\]
is a generic set, for $i=1,\dots,\kappa$. On the set $\cG_i\cap\cG_{\rm reg}$ the condition
${\rm rank\, }Y_i(x)\geq n-\widetilde a_i$ is equivalent to $a_i^{(x)}\leq \widetilde a_i$, and since, by assumption,
the reverse inequality $a_i^{(x)}\geq\widetilde a_i$ holds for all $x\in\cG_{\rm reg}$, it follows that we have
$a_i^{(x)}= \widetilde a_i$ for all $x\in\cG_i\cap\cG_{\rm reg}$. Thus, setting
$\widetilde\cG:=\cG_{\rm reg}\cap\cG_1\cap\dots\cap\cG_\kappa$, we find that $\widetilde\cG$ is generic,
as being the intersection of finitely many generic sets, and for all $x\in\widetilde\cG$ the conditions (a) and (b)
are satisfied.

Finally, let $\chi_x( \la)$ denote the characteristic polynomial of $( L+\Phi_x)(\la)$. Then the number of
distinct roots of $\chi_x$ is given by
\[
{\rm rank}\, S\left(\chi_x,\frac{\partial\chi_x}{\partial  \la}\right) -n +1,
\]
where $S(p_1,p_2)$ denotes the Sylvester resultant matrix {\colb (see, for instance \cite[p. 290]{Barnett90})} of the two polynomials
$p_1(\la)$, $p_2(\la)$.
(Recall that $S(p_1,p_2)$ is a square matrix of size ${\rm deg}\, (p_1)+ {\rm deg}\, (p_2)$ and
that the rank deficiency of $S(p_1, p_2)$ coincides with the degree of the greatest
common divisor of the polynomials $p_1(\la)$ and $p_2(\la)$.)
Therefore, the set $\cG$ of all $x\in\widetilde \cG$ on which the number of distinct roots of $\chi(x)$ is maximal,
is a generic set. (Again this uses \cite[Lemma 2.1]{MehMRR11}, which states that the set where
a matrix depending on $x\in\CC^m$ has maximal rank is a generic set.)
If we can show that this maximal number is equal to $\kappa+\sum_{i=1}^\kappa\frac{1}{\mu}(a_i-\widetilde a_i)$,
then clearly (a)--(c) are satisfied for all $x\in\cG$.
To this end, observe that $P$ as a polynomial is an analytic function and that, by assumption,
$x_{0,i}$ can be chosen to be of arbitrarily small norm. Furthermore, for $\varepsilon_0>0$
sufficiently small, the continuity of $P$ guarantees that, for all $x\in\CC^m$
with $\|x\|\leq\varepsilon_0$, the perturbed pencil $(L+\Phi_x)(\lambda)$ is regular and does not have
$\infty$ as an eigenvalue. But then $B+\Phi_B$ is invertible and we can apply Lemma~\ref{simp}
to the matrix $(B+\Phi_B)^{-1}(A+\Phi_A)$ using the fact that matrix inversion is an analytic function
to prove that the maximal number of distinct roots of $\chi_x$ is as desired.

The additional part follows from the fact that the only list of partial multiplicities that both
dominates $(n_{i,r+1}, \ldots , n_{i,g})$ and has $a_i^{(x)}=n_{i,r+1}+\cdots+n_{i,g_i}$ is the list
$(n_{i,r+1}, \ldots , n_{i,g})$.
\hfill$\Box$

\medskip
The key consequence of Theorem~\ref{local} is the following: If we want to show that a pencil has a
particular behavior under perturbations, it is now enough to consider the pencil locally in the
following sense: it is sufficient to focus on a single eigenvalue
and construct examples of perturbations that provide the desired behavior for that particular
eigenvalue. We will use this strategy exhaustively in the following subsections.

\subsection{Revisiting the unstructured case}

In this subsection, we will briefly revisit the case of \cm{general matrix pencils (possibly without additional symmetry structures)}
and discuss their parameterizations from \cite{dd16}. This will not only give us an idea on how we can extend this
procedure to the case of structured pencils, but also allows us to strengthen the main result in \cite{dd16},
which only considered the generic change in the Weierstra{\ss} structure of regular matrix pencils under low-rank
perturbations, but did not discuss the multiplicity of newly generated eigenvalues.

As in \cite{dd16}, let us pick an integer $r\leq n$ and let us define for each $s=0,1,\dots,r$ the set
\[
\cC_s:=\left\{\begin{array}{c}v_1(\lambda)w_1(\lambda)^\top+\cdots+v_r(\lambda)w_r(\lambda)^\top\end{array}\left|\begin{array}{c}
v_1,\hdots,v_r,w_1,\dots,w_r\in\CC[\la]^n,\\
\deg v_i,\deg w_i\leq1,\ \mbox{\rm for $j=1,\hdots,r$,}\\
\deg v_1=\cdots=\deg v_s=0,\\
\deg w_{s+1}=\cdots=\deg w_r=0
\end{array}
\right\}\right.\,.
\]
Then using \cite[Lemma 2.8]{dd07} it was shown in \cite[Lemma 3.1]{dd16} that
\begin{equation}\label{30.8.18}
\mathbb P_r=\cC_0\cup\cC_1\cup\cdots\cup\cC_r,
\end{equation}
where $\mathbb P_r$ denotes the set of $n\times n$ matrix pencils with rank at most $r$.

\begin{remark}\label{rem:30.8.18}
It is important to note that the union in~\eqref{30.8.18} is not a partition, as the
sets $\cC_0,\cC_1,\dots,\cC_r$ are not disjoint. In particular, if $A\in\mathbb C^{n\times n}$
is a matrix of rank $r$, then the pencil $A=A+\lambda 0$ is contained in each $\cC_s$ for $s=0,\dots,r$.
\end{remark}
\begin{definition}\label{param.def.gen}{\rm (Parameterization of the set of pencils with
rank at most $r$)}. Let $r\in\mathbb N$. \cm{For each $s=0,1,\hdots,r$ we define the map $\Phi_s:\CC^{3rn}\longrightarrow\cC_s$
as follows: for $x\in\CC^{3rn}$ decomposed as $x=\left[\begin{array}{c|c|c|c}\alpha&\beta&\gamma&\delta
\end{array}\right]^\top$ with
\[
\begin{array}{ccc}
\alpha&=&\left[\begin{array}{ccc|c|ccc}\alpha_{11}&\cdots&\alpha_{n1}&\cdots&\alpha_{1r}&\cdots&\alpha_{nr}\end{array}\right]
\in\CC^{1\times rn},\\
\beta&=&\left[\begin{array}{ccc|c|ccc}\beta_{1,s+1}&\cdots&\beta_{n,s+1}&\cdots&\beta_{1r}&\cdots&\beta_{nr}
\end{array}\right]\in\CC^{1\times (r-s)n},\\
\gamma&=&\left[\begin{array}{ccc|c|ccc}\gamma_{11}&\cdots&\gamma_{n1}&\cdots&\gamma_{1r}&
\cdots&\gamma_{nr}\end{array}\right]\in\CC^{1\times r n},\\
\delta&=&\left[\begin{array}{ccc|c|ccc}\delta_{11}&\cdots&\delta_{n1}&\cdots&\delta_{1s}&
\cdots&\delta_{ns}\end{array}\right]\in\CC^{1\times s n},
\end{array}
\]
we set}
\[
\Phi_s(x)=v_1(\lambda)w_1(\lambda)^\top+\cdots+v_r(\lambda)w_r(\lambda)^\top,
\]
where $v_1,\dots,v_r,w_1,\dots,w_r$ are defined via
\[
\begin{array}{rclc}
v_i&=&\left[\begin{array}{ccc}\alpha_{1i}&\cdots&\alpha_{ni}\end{array}\right]^\top,&\mbox{for $i=1,\hdots,s$},\\
v_j&=&\left[\begin{array}{ccc}\alpha_{1j}+\lambda\beta_{1j}&\cdots&\alpha_{nj}+\lambda\beta_{nj}
\end{array}\right]^\top,&\mbox{for $j=s+1,\hdots,r$},\\
w_i&=&\left[\begin{array}{ccc}\gamma_{1i}+\lambda\delta_{1i}&\cdots&\gamma_{ni}+\lambda\delta_{ni}\end{array}\right]^\top,&
\mbox{for $i=1,\hdots,s$},\\
w_j&=&\left[\begin{array}{ccc}\gamma_{1j}&\cdots&\gamma_{nj}\end{array}\right]^\top,&\mbox{for $j=s+1,\hdots,r$}.
\end{array}
\]
\end{definition}
With this preparation, we are able to prove the following result, which extends the main
result from \cite{dd16} by adding a statement on the simplicity of newly generated eigenvalues.
\begin{theorem}\label{gen-main.th}{\rm (Generic change under low-rank perturbations of general regular matrix pencils).}
Let $L(\lambda)$ be a regular $n\times n$  matrix pencil and let $\lambda_1,\dots,\lambda_\kappa$
denote the pairwise distinct eigenvalues of $L(\lambda)$ having the partial multiplicities
$n_{i,1}\geq\dots\geq n_{i,g_i}>0$, for $i=1,\dots,\kappa$, respectively. Furthermore, let $r$ be a positive integer,
let $0\leq s\leq r$, and let $\Phi_s$ be the map in Definition {\rm\ref{param.def.gen}}.
Then, there exists a generic set $\cG_s$ in $\CC^{3rn}$ such that for all
$E(\la)\in\Phi_s(\cG_s)$, the perturbed pencil $L+E$ is regular and the partial multiplicities of
$L+E$ at $\la_i$ are given by $n_{i,r+1}\geq\cdots\geq n_{i,g_i}$.
(In particular, if $r\geq g_i$ then $\lambda_i$ is not an eigenvalue of $L+E$.)
Furthermore, all eigenvalues of $L+E$ that are different from those of $L$ are simple.
\end{theorem}
\begin{proof}
By Theorem~\ref{local} it is sufficient to focus on a particular eigenvalue $\lambda_i$ and construct one particular
example $E=\Phi_s(x)$ of a pencil such that the partial multiplicities of $L+E$ are as claimed in the theorem
and such all eigenvalues that are different from those of $L$ are simple. For the moment, let us suppose that $\lambda_i$
is finite and, for simplicity, let us write $n_1\geq\cdots\geq n_g$ instead of $n_{i,1}\geq\dots\geq n_{i,g_i}$ for its
partial multiplicities. Since genericity of sets is preserved
under multiplication with invertible matrices, we may assume, without loss of generality, that $L$ is in WCF and has the form
\[
L(\lambda)=\diag\big(J_{n_1}(\lambda_i-\lambda),\dots,J_{n_g}(\lambda_i-\lambda),\widetilde L(\lambda)\big),
\]
where $\widetilde L(\lambda)$ consists of all the blocks associated with eigenvalues different from $\lambda_i$.
As in the proof of \cite[Theorem 3.4]{dd16}, let $E_k(\psi)$ be the $k\times k$ matrix that is zero everywhere except
for the $(k,1)$-entry which takes the value $\psi\in\mathbb C$. Then it is straightforward to check that the pencil
$J_m(\lambda_i-\lambda)+E_m(\psi)$ has determinant equal to $\chi(\lambda)=(\lambda_i-\lambda)^m+(-1)^{m-1}\psi$,
i.e., its eigenvalues lie on a circle centered around $\lambda_i$ with radius $|\psi|^{\frac{1}{m}}$. Thus, consider the $n\times n$ pencil
\[
E(\lambda)=\diag\big(E_{n_1}(\psi_1),\dots,E_{n_r}(\psi_r),0).
\]
Then $E(\lambda)$ is a constant pencil of rank $r$ and hence, by Remark~\ref{rem:30.8.18}, there exists $x\in\CC^{3rn}$
such that $E(\lambda)=\Phi_s(x)$. Moreover, we find that $L+E$ has the partial multiplicities $n_{r+1}\geq\cdots\geq n_{g}$ at
$\lambda_i$. Furthermore, having chosen the values $\psi_1,\dots,\psi_r\in\mathbb C$ appropriately such that all radii
$|\psi_j|^{\frac{1}{n_j}}$ are pairwise distinct and smaller than the distance of $\lambda_i$ to the spectrum of $\widetilde L(\lambda)$,
we can guarantee that all eigenvalues of $L+E$ that are different from those of $L$ are simple. Finally, by also choosing
$\psi_1,\dots,\psi_r$ to be of sufficiently small modulus, we can guarantee that the norm of $x$ is arbitrarily small. This
gives the desired example.
For the case $\lambda_i=\infty$ consider the reversal of the pencil $L(\lambda)$ and apply the result for the already proved case
$\lambda_i=0$.
\end{proof}

\subsection{Parameterization of low-rank structured matrix pencils}\label{param.sec}

In this subsection, we finally consider the generic change in the Weierstra{\ss} structure of structured
matrix pencils under structure-preserving low-rank perturbations.
Following the procedure in \cite{dd16}, we first look for a parameterization of the set of $n\times n$
structured matrix pencils with rank at most $r$, for any of the structures considered in Section~\ref{rank1.sec}.
Such a parameterization comes naturally from the decomposition into a sum of rank-$1$ pencils provided
in that section. More precisely, we decompose the set of $n\times n$ structured matrix pencils as the
union of subsets given by fixing the value of the parameter $s$ in Theorems \ref{rank1-herm.th}, \ref{rank1-sym.th},
\ref{rank1-sksym.th}--\ref{rank1-starpal.th}, and~\ref{rank1-starantipal.th}.
Again, we will use the Hermitian case as a model for other structures. Thus, while the Hermitian case
will be presented in full detail, we only give a brief remark on how other structures have to be dealt with
whenever this is necessary, with one exception: we will add a bit more details in the case of $\top$-even pencils,
because the effect of structure-preserving low-rank perturbation needs a more detailed discussion for
this structure and related ones. Thus, the set of $\top$-even pencils will be a subordinate case.

For the Hermitian structure, the decomposition outlined in the previous paragraph is as follows.
For each $0\leq s\leq\lfloor r/2\rfloor$, let us define
\[
\cC_s^\H:=\left\{\begin{array}{c}(a_1+\la b_1)u_1u_1^*+\cdots+(a_\ell+\la b_\ell)u_\ell u_\ell^*\\
+v_1w_1^*+\cdots+v_sw_s^*+w_1v_1^*+\cdots+w_sv_s^*\end{array}\left|\begin{array}{c}
\ell=r-2s,\\
u_1,\hdots,u_\ell\in\CC^n,\\
v_1,\hdots,v_s\in\CC^n,\\
w_1,\hdots,w_s\in\CC[\la]^n,\\
\deg w_j\leq1,\ \mbox{\rm for $j=1,\hdots,s$,}\\
a_i,b_i\in\RR,\  \mbox{\rm for $i=1,\hdots,\ell$}
\end{array}
\right\}\right.\,.
\]
Then, Theorem \ref{rank1-herm.th} states that
\begin{equation}\label{herm-decomp}
\H_r=\cC_0^\H\cup\cC_1^\H\cup\cdots\cup\cC_{\lfloor r/2\rfloor}^\H.
\end{equation}
We emphasize that, as in the general case without particular structure,
the decomposition \eqref{herm-decomp} is not a partition, since the sets
$\cC_i^\H$ are not disjoint.

The case of the structures $Sym_r,S\H_r,Even_r^*,Odd_r^*,Pal_r^*,$ and $Apal_r^*$ is similar and the
decomposition is obtained through the same number of subsets as in \eqref{herm-decomp}, using
\eqref{sym-rank1}, \eqref{sk-rank1}, \eqref{*even-rank1}, \eqref{*odd-rank1}, \eqref{*pal-rank1},
and~\eqref{*antipal-rank1}, respectively, and replacing $*$ by $\top$ and allowing $a_i,b_i\in\CC$
for the case $Sym_r$.

For the remaining structures $SSym_r$, $Even^\top_r$, $Odd^\top_r$, $Pal^\top_r$, and $Apal^\top_r$,
we also have to replace $*$ by $\top$ and allow $a_i,b_i\in\CC$. In addition, the decomposition of the set of
structured matrices of rank $r$ consists of only one set, since the value of $s$ is fixed by $s=r/2$
if $r$ is even, or by $s=(r-1)/2$ if $r$ is odd.

Next, we introduce a parameterization for the sets of $n\times n$ structured matrix pencils
with rank at most $r$ by introducing a parameterization for each of the subsets that give rise
to the decompositions above.

\begin{definition}\label{param.def}{\rm (Parameterization of the set of Hermitian matrix pencils with
rank at most $r$)}. $\;$\\
Let $r\in\mathbb N$. \cm{For each $s=0,1,\hdots,\lfloor r/2\rfloor$ we define the map
$\Phi_s:\RR^{2\ell}\times\CC^{(r+s)n}\longrightarrow\cC_s^\H$ with $\ell=r-2s$ as follows:
For $x\in\CC^{(r+s)n}$ decomposed as $x=\left[\begin{array}{c|c|c|c}\alpha&\beta&\gamma&\delta
\end{array}\right]^\top$ with
\[
\begin{array}{ccc}
\alpha&=&\left[\begin{array}{ccc|c|ccc}\alpha_{11}&\cdots&\alpha_{n1}&\cdots&\alpha_{1\ell}&\cdots&\alpha_{n\ell}\end{array}\right]
\in\CC^{1\times \ell n},\\
\beta&=&\left[\begin{array}{ccc|c|ccc}\beta_{11}&\cdots&\beta_{n1}&\cdots&\beta_{1s}&\cdots&\beta_{ns}
\end{array}\right]\in\CC^{1\times sn},\\
\gamma&=&\left[\begin{array}{ccc|c|ccc}\gamma_{11}&\cdots&\gamma_{n1}&\cdots&\gamma_{1s}&
\cdots&\gamma_{ns}\end{array}\right]\in\CC^{1\times s n},\\
\delta&=&\left[\begin{array}{ccc|c|ccc}\delta_{11}&\cdots&\delta_{n1}&\cdots&\delta_{1s}&
\cdots&\delta_{ns}\end{array}\right]\in\CC^{1\times s n},
\end{array}
\]
we set
\begin{eqnarray*}
&&\Phi_s\left(\left[\begin{array}{ccccc}a_1&b_1&\cdots&a_\ell&b_\ell\end{array}\right]^\top\!\!,\, x\right)\\
&&\qquad=(a_1+\la b_1)u_1u_1^*+\cdots+(a_\ell+\la b_\ell)u_\ell u_\ell^*+v_1w_1^*+\cdots+v_sw_s^*
+w_1v_1^*+\cdots+w_sv_s^*,
\end{eqnarray*}
where $u_1,\dots,u_\ell,v_1,\dots,v_s,w_1,\dots,w_s$ are defined by}
\[
\begin{array}{rclc}
u_i&=&\left[\begin{array}{ccc}\alpha_{1i}&\cdots&\alpha_{ni}\end{array}\right]^\top,&\mbox{for $i=1,\hdots,\ell$},\\
v_j&=&\left[\begin{array}{ccc}\beta_{1j}&\cdots&\beta_{nj}\end{array}\right]^\top,&\mbox{for $j=1,\hdots,s$},\\
\mbox{and}\quad w_j&=&\left[\begin{array}{ccc}\gamma_{1j}+\lambda\delta_{1j}&\cdots&\gamma_{nj}+
\la \delta_{nj}\end{array}\right]^\top,&\mbox{for $j=1,\hdots,s$}.
\end{array}
\]
\end{definition}
\begin{remark}\label{param.rem}
For the other structures, the parameterization is defined analogously. More precisely, let $\sym_r$ be the set
of $n\times n$ matrix pencils with rank at most $r$ having the structure $\sym$ and assume that
$\sym_r=\cC_{i_1}^{\sym}\cup\hdots\cup\cC_{i_k}^\sym$ is a decomposition {\colb into smaller subsets}, where the number $k$
depends on the structure and on $r$. Then the parameterization of $\sym_r$ is a tuple
of continuous, surjective maps $\Phi_s:\RR^{p_s}\times\CC^{m_s}\longrightarrow\cC_s^\sym$, {\colb for $s\in\{i_1,\hdots,i_k\}$,}
and {\colb where} $p_s,m_s$ depend on $s$. (In fact, these parameterizations are not only continuous, but are
polynomials either in the entries of $x$ or in the real and imaginary parts of the entries of $x$.)

For the Hermitian, skew-Hermitian, $*$-even, $*$-odd, $*$-palindromic, and
$*$-anti-palindromic structures, we have $k=\lfloor r/2\rfloor+1$, $\{i_1,\hdots,i_k\}=\{0,1,\hdots,\lfloor r/2\rfloor\}$,
$p_s=2(r-2s)$, and $m_s=(r+s)n$, while for the symmetric structure, we have $k=\lfloor r/2\rfloor+1$,
$\{i_1,\hdots,i_k\}=\{0,1,\hdots,\lfloor r/2\rfloor\}$, $p_s=0$, and $m_s=2(r-2s)+(r+s)n$.

In the remaining structures, we have $k=1,s=\lfloor r/2\rfloor$, $p_s=0$, and
$m_s=\lfloor 3r/2\rfloor n$. For example, for the case of $\top$-even pencils,
the map
\begin{equation}\label{phione}
\Phi:\CC^{\lfloor\frac{3r}{2}\rfloor n}\longrightarrow Even^\top_r
\end{equation}
is defined by $\Phi(x)=E(\la)$, with $E(\la)$ as in \eqref{teven-rank1}, and
where $u,v_j,w_j$, for $j=1,\hdots,\lfloor r/2\rfloor$, are defined as follows: if
$x\in\CC^{\lfloor 3r/2\rfloor n}$ is decomposed as
$x=\left[\begin{array}{c|c|c|c}
\alpha&\beta&\gamma&\delta
\end{array}\right]^\top,$
where
\[
\begin{array}{ccl}
\alpha&=&\left[\begin{array}{ccc}\alpha_1&\cdots&\alpha_{\ell n}\end{array}\right]\in\CC^{1\times\ell n},\\
\beta&=&\left[\begin{array}{ccc|c|ccc}\beta_{11}&\cdots&\beta_{n1}&\cdots&\beta_{1, \lfloor r/2\rfloor n}
&\cdots&\beta_{n, \lfloor r/2\rfloor n}\end{array}\right]\in\CC^{1\times \lfloor r/2\rfloor n},\\
\gamma&=&\left[\begin{array}{ccc|c|ccc}\gamma_{11}&\cdots&\gamma_{n1}&\cdots&\gamma_{1, \lfloor r/2\rfloor n}
&\cdots&\gamma_{n, \lfloor r/2\rfloor n}\end{array}\right]\in\CC^{1\times  \lfloor r/2\rfloor n},\\
\delta&=&\left[\begin{array}{ccc|c|ccc}\delta_{11}&\cdots&\delta_{n1}&\cdots&\delta_{1, \lfloor r/2\rfloor n}
&\cdots&\delta_{n, \lfloor r/2\rfloor n}\end{array}\right]\in\CC^{1\times \lfloor r/2\rfloor n},
\end{array}
\]
with $\ell=r-2\lfloor r/2\rfloor$, then
\[
\begin{array}{cclc}
u&=&\left[\begin{array}{ccc}\alpha_{1}&\cdots&\alpha_{\cm{\ell n}}\end{array}\right]^\top,&\\
v_j&=&\left[\begin{array}{ccc}\beta_{1j}&\cdots&\beta_{nj}\end{array}\right]^\top,&\mbox{for $j=1,\hdots,\lfloor r/2\rfloor$},\\
w_j&=&\left[\begin{array}{ccc}\gamma_{1j}+\lambda\delta_{1j}&\cdots&\gamma_{nj}+\la \delta_{nj}\end{array}\right]^\top,
&\mbox{for $j=1,\hdots,\lfloor r/2\rfloor$}.
\end{array}
\]
Note that $\alpha$ is void if $r$ is even, because we then have $\ell=0$.

{\colb We highlight that, in all cases, the map $\Phi_s$ is surjective.}
\end{remark}
%

\subsection{Generic perturbation theory for pencils with symmetry structures}\label{pert.sec}

In this subsection, we will develop the eigenvalue perturbation theory of regular matrix pencils
with symmetry structures under structure-preserving perturbations with the help of the
parameterizations from Section~\ref{param.sec}.
The sets of the form $\mathbb R^{p_s}\times\mathbb C^{m_s}$ that appear as domains for
the parameterizations constructed analogous to Definition~\ref{param.def} will be identified
with the set $\mathbb R^{p_s+2m_s}$ by splitting the variables in $\mathbb C$ into their
real and imaginary parts. \vm {As noted before, this detour via the reals is necessary when symmetry structures involving complex conjugation are considered.}
When we deal with symmetry structures only involving the complex transpose, but not complex conjugation,
then we have $p_s=0$ and \cm{we can express genericity in terms of complex polynomials only.}
\begin{theorem}\label{herm-main.th}{\rm (Generic change under low-rank perturbations of Hermitian pencils).}
Let $L(\lambda)$ be a regular $n\times n$ Hermitian matrix pencil and let $\lambda_1,\dots,\lambda_\kappa$
denote the pairwise distinct eigenvalues of $L(\lambda)$ having the partial multiplicities
$n_{i,1}\geq\dots\geq n_{i,g_i}>0$ for $i=1,\dots,\kappa$, respectively. Furthermore, let $r$ be a positive integer,
let $0\leq s\leq\lfloor r/2\rfloor$, and let $\Phi_s$ be the map in Definition {\rm\ref{param.def}}
and $\ell=r-2s$.
Then, there exists a generic set $\cG_s$ in $\mathbb R^\cm{2\ell}\times\CC^{(r+s)n}$ such that, for all
$E(\la)\in\Phi_s(\cG_s)$, the perturbed pencil $L+E$ is regular and the partial multiplicities of
$L+E$ at $\la_i$ are given by $n_{i,r+1}\geq\cdots\geq n_{i,g_i}$.
(In particular, if $r\geq g_i$ then $\lambda_i$ is not an eigenvalue of $L+E$.)
Furthermore, all eigenvalues of $L+E$ that are different from those of $L$ are simple.
\end{theorem}
\begin{proof}
By Theorem~\ref{local} \cm{(applied for the case $\mathbb F=\mathbb R$ and $m=2\ell+2(r+s)n$ in accordance with
the identification $\mathbb R^{2\ell+2(r+s)n}=\mathbb R^{2\ell}\times\mathbb C^{(r+s)n}$)}
it is sufficient to show, for each $i=1,\dots,\kappa$, the existence of
one particular $x_i\in\mathbb R^\cm{2\ell}\times\mathbb C^{(r+s)n}$ of arbitrarily small norm such that, with
the corresponding perturbation pencil $E(\la)=\Phi_s(x_i)$,
the perturbed pencil $L+E$ has precisely the partial multiplicities $n_{i,r+1}\geq\cdots\geq n_{i,g_i}$ at $\la_0$ and all
eigenvalues of $L+E$ that are different from those of $L$ are simple.
Since genericity of sets is invariant under multiplication with invertible matrices, it suffices
to consider the case when $L$ is given in Hermitian canonical form (Theorem \ref{thm:canformherm}). 
To this end, we distinguish three cases and for the ease of notation we will from now on
drop the dependence on $i$ of the geometric multiplicity and partial multiplicities of $\lambda_i$, thus writing $g$ and
$n_1,\dots,n_g$ instead of $g_i$ and $n_{i,1},\dots,n_{i,g_i}$.

\emph{Case (1)}: $\la_i\in\RR$. Then we can assume, without loss of generality, that $L$ is of the form
\[
L(\la)=\diag\big(\sigma_1RJ_{n_1}(\lambda_i-\la),\hdots,\sigma_gRJ_{n_g}(\lambda_i-\la),\widetilde L(\la)\big),
\]
where $\lambda_i$ is not an eigenvalue of $\widetilde L(\la )$.
Let $F_\nu=\widetilde u\widetilde u^*$, with $\widetilde u=e_1\in\mathbb C^\nu$, and
$G_{\nu,\widetilde \nu}=\widetilde v\widetilde w^*+\widetilde w\widetilde v^*$, with
$\widetilde v=e_{\widetilde \nu+1},\widetilde w=\frac{1}{2}e_1\in\mathbb C^{\nu+\widetilde \nu}$, i.e.,
$F_\nu$ is the $\nu\times\nu$ matrix
that is everywhere zero except for $F_\nu(1,1)=1$, and $G_{\nu,\widetilde \nu}$ is the
$(\nu+\widetilde \nu)\times(\nu+\widetilde \nu)$ matrix which is
everywhere zero except for $G_{\nu,\widetilde \nu}(1,\widetilde \nu+1)=G_{\nu,\widetilde \nu}(\widetilde
\nu+1,1)=1$. Note that both $F_\nu$ and $G_{\nu,\widetilde \nu}$ are
Hermitian matrices.

First, let us assume that $r\leq g$. Then, \cm{we set}
\begin{equation}\label{eq:25.7.18}
E(\la)=\diag(\alpha_1F_{n_1},\hdots,\alpha_\ell F_{n_\ell},
\beta_1G_{n_{\ell+1},n_{\ell+2}},\hdots,\beta_sG_{n_{r-1},n_{r}},0)+\la 0_{n\times n}
\end{equation}
for some values $\alpha_1,\dots,\alpha_\ell,\beta_1,\dots,\beta_s\in\mathbb R$ to be specified later.
The matrix pencil $E(\la)$ has rank $r$ and, from the construction of $F_m$ and
$G_{m,\widetilde m}$, it is clear that $E(\lambda)$ can be written in the form~\eqref{herm-rank1}
(\cm{e.g.,} with $a_1=\hdots=a_\ell=1,b_1=\hdots=b_\ell=0$). Thus, we have
$E(\la)\in\cC_s^\H$. Then, since $\Phi_s$ is surjective, there exists some \cm{$x\in\mathbb R^{2\ell}\times\CC^{(r+s)n}$} such that
$\Phi_s(x)=E(\la)$, \cm{and provided that the parameters $\alpha_i,\beta_j$ are sufficiently small, it is clear that this
$x$ can be chosen to be of arbitrarily small norm. (This uses the fact that  $\Phi_s$ is not injective, i.e.,
we can ``split up'' the small values $\alpha_i,\beta_j$ and put them into the parameters $a_i,b_i,u_j,v_k,w_k$ of Definition~\ref{param.def}
in such a way that all entries of $x$ are small.)} Moreover, the nonzero partial multiplicities of $L+E$ at $\lambda_i$ are
$(n_{r+1},\hdots, n_g)$. To see this, note first that only the first $r$ blocks of $L$ are modified so, in particular, $L+E$
contains $g-r$ Jordan blocks associated with $\lambda_i$ with sizes $(n_{r+1},\hdots, n_g)$.
(If $g=r$, then this means that $\lambda_i$ is not an eigenvalue of $L+E$.)
Furthermore, the part of the pencil $L+E$ corresponding {\colb to} the first $r$ blocks of $L$ is block diagonal,
and with the help of the Laplace expansion it is easy to verify that the characteristic
polynomials of its diagonal blocks $RJ_{n_j}(\lambda_i-\lambda)+\alpha_jF_{n_j}$, $j=1,\dots,\ell$, and
$\diag(RJ_{n_{\ell+2j-1}}(\lambda_i-\lambda),RJ_{n_{\ell+2j}}(\lambda_i-\lambda))+\beta_jG_{n_{\ell+2j-1},n_{\ell+2j}}$, for
$j=1,\dots,s$, are given by
\[
(-1)^{\varrho_j}\big((\lambda-\lambda_i)^{n_j}-\alpha_j\big),\;j=1,\dots,\ell\quad\mbox{and}\quad
(-1)^{\varrho_{\ell+j}}\big((\lambda-\lambda_i)^{n_{\ell+2j-1}+n_{\ell+2j}}-\beta_j^2\big),\;j=1,\dots,s,
\]
respectively,
where $\varrho_1,\dots,\varrho_{\ell+s}$ are integers only depending on the sizes $n_1,\dots,n_r$
and the signs $\sigma_1,\dots,\sigma_r$.
Thus, the eigenvalues of this diagonal blocks lie on circles centered around $\lambda_i$ with
radii $|\alpha_1|^{\frac{1}{n_1}},\dots,|\alpha_\ell|^{\frac{1}{n_\ell}}$,
$|\beta_1|^{\frac{2}{n_{\ell+1}+n_{\ell+2}}},\dots,|\beta_s|^{\frac{2}{n_{r-1}+n_{r}}}$.
Clearly, choosing the parameters $\alpha_1,\dots,\alpha_\ell$ and $\beta_1,\dots,\beta_s$ appropriately,
we can guarantee that all eigenvalues of $L+E$ that are different from those of $L$ are simple.

Now assume that $g<r$. If $g\leq \ell$ or if $g$ has the same parity as $\ell$ (i.e. $g-\ell$ is even)
then we define $E(\lambda)$ as in~\eqref{eq:25.7.18}, where we interpret $n_j=0$ for $j>g$.
Then $E(\lambda)$ has rank less than $r$, but still can be written in the form~\eqref{herm-rank1}.
Indeed, if $g\leq\ell$ then we set $u_i=0$ for $i>g$ and $v_j=w_j=0$ for $j=1,\dots,s$,
and if $g>\ell$ then we set $v_j=w_j=0$ for $j=\frac{g-\ell}{2}+1,\dots,s$.
If, on the other hand, $g>\ell$ and $g$ has the opposite parity to $\ell$, i.e. $g-\ell=2\kappa+1$,
then we slightly alter the pencil in~\eqref{eq:25.7.18} to
%
\[
E(\la)=\diag(\alpha_1F_{n_1},\hdots,\alpha_\ell F_{n_\ell},
\beta_1G_{n_{\ell+1},n_{\ell+2}},\hdots,\beta_\kappa G_{n_{\ell+2\kappa-1},n_{\ell+2\kappa}},
\beta_{\kappa+1}F_{n_g},0)+
\la 0_{n\times n}.
\]
%
Also this pencil can be written in the form~\eqref{herm-rank1}, noting that a block
$F_\nu$ can also be represented in the form $\widetilde v\widetilde w^*+\widetilde w\widetilde v^*$
by choosing $\widetilde v=\widetilde w=\frac{1}{2}e_1$. In all cases, the perturbed pencil $L+E$
does not have the eigenvalue $\lambda_i$ and all eigenvalues different from those of $L$ are simple
if the parameters $\alpha_i$ and $\beta_j$ are chosen appropriately.

\emph{Case (2)}: $\la_i=\infty$. This case follows by applying the already proved Case (1) to the
reversal of the pencil $L$.

\emph{Case (3)}: $\lambda_i\in\mathbb C\setminus\mathbb R$. In the following we denote $\lambda_i$
by $\mu$, for consistency with the notation used before.
In this case, the Hermitian canonical form contains $2k\times 2k$ coupled blocks associated with $\mu$
and $\overline \mu$, each of size $k\times k$, as indicated in the proof of Theorem \ref{rank1-herm.th}.
Then, we may assume that $L(\la)$ is of the form
\[
\begin{array}{cl}
L(\la)=&\diag\left(R\diag(J_{n_1}(\overline\mu-\la),J_{n_1}(\mu-\la)),\hdots,\right.\\&
\textcolor{white}{\diag(}\left.R
\diag(J_{n_g}(\overline\mu-\la),J_{n_g}(\mu-\la)),\widetilde L(\la))\right),
\end{array}
\]
where, again, neither $\mu$ nor $\overline\mu$ are eigenvalues of $\widetilde L(\la)$. Furthermore,
we assume that $g\geq r$. (The subcase $g<r$ can be treated analogously to the corresponding subcase in
Case~(1).)

Let $\widetilde F_{2\nu}=uu^*$, with $u=e_1+e_{\nu+1}\in\mathbb C^{2\nu}$ and
$\widetilde G_{2\nu,2\widetilde\nu}=vw^*+wv^*$, with $v=e_{2\nu+1}+e_{2\nu+\widetilde\nu+1}\in\CC^{2(\nu+\widetilde\nu)}$,
$w=\frac{1}{2}(e_1+e_{\nu+1})\in\CC^{2(\nu+\widetilde\nu)}$. Thus $\widetilde F_{2\nu}$ is the $2\nu\times 2\nu$ matrix whose
entries are all zero except for the entries in the positions $(1,1)$, $(1,\nu+1)$, $(\nu+1,1)$
and $(\nu+1,\nu+1)$, which are all equal to $1$, and
$\widetilde G_{2\nu,2\widetilde\nu}$ is the $2(\nu+\widetilde\nu)\times 2(\nu+\widetilde\nu)$ matrix whose
entries are all zero except for the entries in the positions $(1, 2\nu+1)$, $(1,2\nu+\widetilde\nu+1)$,
$(\nu+1,2\nu+1)$, $(\nu+1,2\nu+\widetilde\nu+1)$, $(2\nu+1,1)$, $(2\nu+1,\nu+1)$,
$(2\nu+\widetilde\nu+1,1)$, and $(2\nu+\widetilde\nu+1,\nu+1)$ which are all equal to $1$.
\cm{Let $E(\la)$ be}
\begin{equation}\label{herm-pert}
E(\la)=\diag(\alpha_1\widetilde F_{2n_1},\hdots,\alpha_\ell\widetilde F_{2n_\ell},\beta_1\widetilde G_{2n_{\ell+1},2n_{\ell+2}},
\hdots,\beta_s\widetilde G_{2n_{r-1},2n_{r}},0)+\la 0_{n\times n},
\end{equation}
where the real parameters $\alpha_1,\dots,\alpha_\ell,\beta_1,\dots,\beta_s$ will be specified later.

By construction, $\rank E=r$ and $E(\la)\in\cC_s^\H$. Again, since $\Phi_s$ is surjective, there is
some $x\in\cm{\mathbb R^{2\ell}\times\CC^{(r+s)n}}$ such that $\Phi_s(x)=E(\la)$. \cm{(Again, $x$ can be chosen to be of arbitrarily
small norm provided that the parameters $\alpha_i,\beta_j$ are sufficiently small.)}
It remains to see that the partial multiplicities
of $L+E$ at $\mu$ are $(n_{r+1}, \hdots,n_g)$ and that all eigenvalues of $L+E$ that are different from
those of $L$ are simple. Again, since the smallest $g-r$ Jordan blocks associated
with $\mu$ in $L(\la)$ are not modified by the perturbation $E(\la)$, they will stay in the WCF of $L+E$,
so $(n_{r+1}, \hdots,n_g)$ is a sublist of the list of partial multiplicities of $L+E$ at $\la_0$.

With the help of the Laplace expansion, one can easily show that the determinant of
each block $R\diag(J_{n_i}(\overline\mu-\la),J_{n_i}(\mu-\la))+\alpha_i\widetilde F_{n_i,n_i}$ is given by
\[
\chi_i(\lambda)=(-1)^{\varrho_i}\big((\lambda-\mu)^{n_i}(\lambda-\overline\mu)^{n_i}-
\alpha_i(\lambda-\mu)^{n_i}-\alpha_i(\lambda-\overline\mu)^{n_i}\big),
\]
where $\varrho_i$ is an integer only depending on $n_i$. It was shown in \cite[Example 4.2]{MehMRR12}
that such a polynomial has simple roots (and clearly these are different from $\mu$ and $\overline\mu$)
if $\alpha_i$ is chosen such that $|\alpha_i|\leq\frac{|\mu-\overline\mu|^{n_i}}{2}$.

On the other hand, again with the help of the Laplace expansion and performing tedious but elementary
calculations, one finds that the determinant of each block
\[
R\diag(J_{n_{\ell+2j-1}}(\overline\mu-\la),J_{n_{\ell+2j-1}}(\mu-\la),J_{n_{\ell+2j}}(\overline\mu-\la),
J_{n_{\ell+2j}}(\mu-\la))+\beta_jG_{n_{\ell+2j-1},n_{\ell+2j}}
\]
is given by
\begin{eqnarray*}
&&\chi_{\ell+j}(\lambda)\\
&=&(-1)^{\varrho_j}\big(
(\lambda-\mu)^{n_{\ell+2j-1}+n_{\ell+2j}}(\lambda-\overline\mu)^{n_{\ell+2j-1}+n_{\ell+2j}}
-\beta_j^2(\lambda-\mu)^{n_{\ell+2j-1}}(\lambda-\overline\mu)^{n_{\ell+2j}}
\\
&&\textcolor{white}{(-1)^{\varrho_j}\big(}
-\beta_j^2(\lambda-\mu)^{n_{\ell+2j}}(\lambda-\overline\mu)^{n_{\ell+2j-1}}
-\beta_j^2(\lambda-\mu)^{n_{\ell+2j-1}+n_{\ell+2j}}
-\beta_j^2(\lambda-\overline\mu)^{n_{\ell+2j-1}+n_{\ell+2j}}\big).
\end{eqnarray*}
If $|\beta_j|$ is sufficiently small, then $\chi_{\ell+j}$ is guaranteed to have
only simple roots (that are clearly all different from $\mu$ and $\overline\mu$).
Indeed, assume that $\lambda$ is a common root of $\chi_{\ell+j}$ and $\chi_{\ell+j}'$.
Then multiplying the equation $\chi_{\ell+j}=0$ with $(\lambda-\mu)(\lambda-\overline\mu)$
and using twice the equation $\chi_{\ell+j}(\lambda)=0$, we obtain that
\[
\beta^2\big((\lambda-\mu)^{n_{\ell+2j-1}+n_{\ell+2j}}+(\lambda-\overline\mu)^{n_{\ell+2j-1}+n_{\ell+2j}}\big)=0,
\]
which implies $|\lambda-\mu|=|\lambda-\overline\mu|$. Using the fact that roots of polynomials
depend continuously on the coefficients of the polynomials it follows that $\beta_j$ can be chosen
sufficiently small such that the roots of $\chi_{\ell+j}$ have a distance from either $\mu$ or
$\overline\mu$ less than $\frac{|\mu-\overline\mu|}{2}$ which then contradicts
$|\lambda-\mu|=|\lambda-\overline\mu|$.

Therefore, choosing $\alpha_1,\dots,\alpha_\ell,\beta_1,\dots,\beta_s$ appropriately, we can
guarantee that there are $n_1+\dots+n_r$ simple eigenvalues close to $\mu$ or $\overline\mu$, respectively,
corresponding to the $r$ Jordan blocks that were perturbed by $E$. Indeed, after having chosen
$\alpha_1$, let $\delta_1$ denote the smallest distance of a root of $\chi_1$ to the set $\{\mu,\overline\mu\}$.
Then choose $\alpha_2$ so small that the (simple) roots of $\chi_2$ are located within circles of
a radius less then $\delta_1$ around $\mu$ or $\overline\mu$, respectively. Then let $\delta_2$ be
the smallest distance of a root of $\chi_2$ to the set $\{\mu,\overline\mu\}$ and continue in this
manner choosing $\alpha_3,\dots,\alpha_\ell,\beta_1,\dots,\beta_s$ such that all eigenvalues of
$L+E$ that are different from the eigenvalues of $L$ are simple.
\end{proof}

\begin{theorem}\label{sym-main.th}{\rm (Generic change under low-rank perturbations of symmetric pencils).}
Let $L(\lambda)$ be a regular $n\times n$ symmetric matrix pencil and let $\lambda_1,\dots,\lambda_\kappa$
denote the pairwise distinct eigenvalues of $L(\lambda)$ having the partial multiplicities
$n_{i,1}\geq\dots\geq n_{i,g_i}>0$ for $i=1,\dots,\kappa$, respectively. Furthermore, let $r$ be a positive integer,
let $0\leq s\leq\lfloor r/2\rfloor$ and let $\Phi_s$ be the map as in Remark~{\rm\ref{param.rem}}
and $\ell=r-2s$.
Then there exists a generic set $\cG_s$ in $\cm{\CC^{2\ell+(r+s)n}}$ such that, for all
$E(\la)\in\Phi_s(\cG_s)$, the perturbed pencil $L+E$ is regular and the partial multiplicities of
$L+E$ at $\la_i$ are given by $n_{i,r+1}\geq\cdots\geq n_{i,g_i}$.
(In particular, if $r\geq g_i$ then $\lambda_i$ is not an eigenvalue of $L+E$.)
Furthermore, all eigenvalues of $L+E$ that are different from those of $L$ are simple.
\end{theorem}
\begin{proof} The proof is similar to the one of Theorem \ref{herm-main.th} \cm{now applying Theorem~\ref{local} for the
case $\mathbb F=\mathbb C$ and $m=2\ell+(r+s)n$}. The only difference
comes from the blocks in the symmetric canonical form, which are different to the ones in the Hermitian
canonical form. In particular, in the symmetric case there is no need to distinguish between real and
complex eigenvalues, so we can follow exactly the same arguments as in the proof of Theorem \ref{herm-main.th}
for an eigenvalue $\la_i\in\mathbb R$, which now is valid for a general $\la_i\in\CC$.
\end{proof}
\begin{theorem}\label{talter-main.th}{\rm (Generic change under low-rank perturbations of $\top$-alternating pencils).}
Let $L(\lambda)$ be a regular $n\times n$ $\top$-alternating matrix pencil and let $\lambda_1,\dots,\lambda_\kappa$
denote the pairwise distinct eigenvalues of $L(\lambda)$ having the partial multiplicities
$n_{i,1}\geq\dots\geq n_{i,g_i}>0$ for $i=1,\dots,\kappa$, respectively. Furthermore, let $r$ be a positive integer
and let $\Phi$ be the map as in Remark~{\rm\ref{param.rem}}, i.e.,
$\Phi$ is as in~\eqref{phione}.
Then, there exists a generic set $\cG$ in $\CC^{\lfloor\frac{3r}{2}\rfloor n}$ such that for all
$E(\la)\in\Phi(\cG)$, the perturbed pencil $L+E$ is regular and the partial multiplicities of
$L+E$ at $\la_i$ are the ones given in
Table {\rm\ref{alter.table}}, where \eqref{property} is the following property:
\begin{equation}\label{property}\tag{P}
n_{i,r}=n_{i,r+1}=\cdots=n_{i,r+d}>n_{i,r+ d+1},\qquad\mbox{\rm with $d$ odd}.
\end{equation}
\begin{center}
\begin{table}[htb!]\footnotesize
\begin{tabular}{|c|c|l|c|}\hline
Structure&e-val $\la_i$&case&multiplicities\\\hline\hline
\multirow{7}{*}{$\top$-even}&
\multirow{2}{*}{$\la_i=0$}&$n_{i,r+1}$ odd and \eqref{property} holds&$\cm{(n_{i,r+1}+1,n_{i,r+2},\dots,n_{i,g_i})}$\\
&& otherwise&$\cm{(n_{i,r+1},n_{i,r+2},\dots,n_{i,g_i})}$\\\cline{2-4}    
&\multirow{4}{*}{$\la_i=\infty$}& $r$ even, $n_{i,r+1}$ even, and \eqref{property} holds&$\cm{(n_{i,r+1}+1,n_{i,r+2},\dots,n_{i,g_i})}$\\
&& $r$ even, otherwise&$\cm{(n_{i,r+1},n_{i,r+2},\dots,n_{i,g_i})}$\\
&&$r$ odd, $n_{i,r+1}$ even, and \eqref{property} holds&$\cm{(n_{i,r+1}+1,n_{i,r+2},\dots,n_{i,g_i},1)}$\\
&& $r$ odd, otherwise&$\cm{(n_{i,r+1},n_{i,r+2},\dots,n_{i,g_i},1)}$\\\cline{2-4}     
&$\la_i\in\CC\setminus\{0\}$&all&$\cm{(n_{i,r+1},n_{i,r+2},\dots,n_{i,g_i})}$\\\hline\hline
\multirow{7}{*}{$\top$-odd}&
\multirow{4}{*}{$\la_i=0$}& $r$ even, $n_{i,r+1}$ even, and \eqref{property} holds&$\cm{(n_{i,r+1}+1,n_{i,r+2},\dots,n_{i,g_i})}$\\
&& $r$ even, otherwise&$\cm{(n_{i,r+1},n_{i,r+2},\dots,n_{i,g_i})}$\\
&& $r$ odd, $n_{i,r+1}$ even, and \eqref{property} holds&$\cm{(n_{i,r+1}+1,n_{i,r+2},\dots,n_{i,g_i},1)}$\\
&& $r$ odd, otherwise&$\cm{(n_{i,r+1},n_{i,r+2},\dots,n_{i,g_i},1)}$\\\cline{2-4}     
&\multirow{2}{*}{$\la_i=\infty$}&$n_{i,r+1}$ odd and \eqref{property} holds&$\cm{(n_{i,r+1}+1,n_{i,r+2},\dots,n_{i,g_i})}$\\
&& otherwise&$\cm{(n_{i,r+1},n_{i,r+2},\dots,n_{i,g_i})}$\\\cline{2-4}     
&$\la_i\in\CC\setminus\{0\}$&all&$\cm{(n_{i,r+1},n_{i,r+2},\dots,n_{i,g_i})}$\\\hline
\end{tabular}
\caption{\footnotesize Generic partial multiplicities at $\la_i$ for rank-$r$ $\top$-alternating perturbations}\label{alter.table}
\end{table}
\end{center}
In particular, if $r\geq g_i$ then $\lambda_i$ is not an eigenvalue of $L+E$.
Furthermore, all eigenvalues of $L+E$ that are different from those of $L$ are simple.
\end{theorem}
\begin{proof}
For simplicity, we drop the dependence on $i$ in the geometric and partial multiplicities of $\lambda_i$,
i.e., we write $g$ instead of $g_i$ and $n_1\geq\dots\geq n_g$ instead of $n_{i,1}\geq\dots\geq n_{i,g_i}$.
We also replace $\la_i$ by $\la_0$.
We will only prove the case $g\geq r$ in full detail. (The case $g<r$ can be treated similarly by constructing
an analogous perturbation of rank $g$ instead of rank $r$, thus showing that $\lambda_i$ is not an eigenvalue
of the perturbed pencil.)
\cm{We aim to apply Theorem~\ref{local} for the case $\mathbb F=\mathbb C$ to any single eigenvalue of
the pencil. Here we make use of the fact that, in contrast to the Hermitian case, the set $Even^\top_r$ need not be
decomposed into smaller sets that can be parameterized as in the sense of Definition~\ref{param.def},
but the parameterization map $\Phi$ as in~\eqref{phione} is already a map onto $Even^\top_r$.}

\cm{\emph{Case 1): property~\eqref{property} does not apply.}}
\cm{We} first consider all cases except those where property \eqref{property} appears in Table \ref{alter.table}.
In these cases, it is sufficient to prove the existence of one particular perturbation $E(\la)$ of
arbitrarily small norm which belongs to $Even^\top_r$.

\cm{\emph{Subcase 1a): $\lambda_0\in\mathbb C\setminus\{0\}$.}}
As in the proof of Theorem \ref{herm-main.th}, we may assume that $L(\la)$ is given in $\top$-alternating
canonical form. \cm{Let us start with the $\top$-even structure.}
In the $\top$-even canonical form, the blocks associated with $\la_0$ and $-\la_0$ appear in
pairs \cite[Th. 2.16]{batzke-thesis}. Then, we may assume that $L(\la)$ is of the form:
\[
\begin{array}{ccl}
L(\la)&=&\diag\big(R\diag(-\la I-J_{n_1}(\la_0),\la I-J_{n_1}(\la_0)),\hdots,\\
&&\textcolor{white}{\diag\big(}R\diag(-\la I-J_{n_g}(\la_0),\la I-J_{n_g}(\la_0)),\widetilde L(\la)\big),
\end{array}
\]
where $\la_0$ is not an eigenvalue of $\widetilde L(\la)$.

Let $\widetilde F_{2m}$ and $\widetilde G_{2m,2n}$ be the same matrices as in the proof of
Theorem \ref{herm-main.th}, and let $E(\la)$ be the pencil in \eqref{herm-pert}. Note that the
pencil $E(\la)$ belongs to $Even^\top_r$.
Therefore, there is some $x\in\CC^{\lfloor\frac{3r}{2}\rfloor n}$ such that $\Phi(x)=E(\la)$, \cm{and $x$ can be chosen
to be of arbitrarily small norm provided that the parameters $\alpha_i,\beta_j$ are sufficiently small.}
Moreover, with similar reasonings to the ones in the proof of Theorem \ref{herm-main.th}, it can
be seen that the nonzero partial multiplicities at $\la_0$ in $L+E$ are $(n_{r+1},\hdots,n_g)$,
and that all eigenvalues of $L+E$ different from those of $L$ are simple, if the parameters
$\alpha_i,\beta_j$ in the pencil~\eqref{herm-pert} have been chosen appropriately.

The case of the $\top$-odd structure can be addressed in a similar way,
just multiplying by $\la$ the perturbation blocks $\widetilde F_{2m}$ and $\widetilde G_{2m,2n}$
in \eqref{herm-pert}.

\cm{\emph{Subcase 1{\colb b}): $\lambda_0=0$ and $\top$-even structure.}}
\cm{Recall that by assumption} condition \eqref{property} is not satisfied.
Then $L(\la)$ is of the form
\[
L(\la)=\diag(L_0(\la),\widehat L_0(\la),\widetilde L(\la)),
\]
where $L_0(\la)$ contains the Jordan blocks corresponding to the largest $r$ partial multiplicities
at $0$ (namely, $n_1\geq\cdots\geq n_r$), $\widehat L_0$ contains the blocks corresponding to the
remaining partial multiplicities at $0$, and $\widetilde L(\la)$ contains the information of the
nonzero eigenvalues.

If $n_{r+1}$ is even or $n_{ r+1}$ is odd, but $n_r=n_{r+1}=\cdots=n_{r+ d}>n_{r+d+1}$ with $d$ even
(i.e., (P) does not hold), then
the part $L_0(\la)$ is a direct sum of blocks of two types:

\begin{itemize}
\item[(i)] a $2k\times 2k$ block of the form
\[
\left[\begin{array}{cccccc}
&&&&&\la\\
&&&&\iddots&1\\
&&&\la&\iddots\\
&&-\la&1\\
&\iddots&\iddots\\
-\la&1
\end{array}\right]_{2k\times 2k}.
\]
\item[(ii)] A pair of $(2k+1)\times(2k+1)$ blocks of the form $R\diag (J_{2k+1}(-\la),J_{2k+1}(\la))$.
\end{itemize}
This is a consequence of the fact that, in the $\top$-even canonical form, the Jordan blocks with odd size associated
with the eigenvalue $0$ are paired up, and can be matched up to form pairs as in blocks of the form (ii)
(see \cite[Th. 2.16]{batzke-thesis}). Therefore, the blocks in $L_0(\la)$ with odd size larger than $n_r$
(if any) are paired up, and, since $d$ is even, also those of size $n_r$ (if any) are paired up.

For each block of type (i) we can add a rank-$1$ perturbation by adding just one entry equal to $\alpha$
in the upper left corner of the block. This perturbation is of the form $uu^\top$ (actually, it is
$\alpha F_{2k}$ in the proof of Theorem \ref{herm-main.th}), and it is easily checked that the characteristic
polynomial of the resulting perturbed block is given by $\chi=\lambda^{2k}-(-1)^k\alpha$ which means that
its eigenvalues are simple and on a circle with center in the origin and radius $|\alpha|^{\frac{1}{2k}}$.
For each pair of blocks of type (ii) we can add a rank-$2$ perturbation by adding entries equal to $\beta$ in
the positions $(1,1)$ and $(2k+2,2k+2)$. This perturbation is of the form
$\beta(vv^\top+ww^\top)$ with $v=e_1$ and $w=e_{2k+2}$, and, again, it is easily checked that the characteristic
polynomial of the resulting perturbed block is given by $\chi=\lambda^{4k+2}+\beta^2$ which implies that
its eigenvalues are simple and on a circle with center in the origin and radius $|\beta|^{\frac{1}{2k+1}}$.
Therefore, choosing the parameters $\alpha$ and $\beta$ appropriately,
we can construct a rank-$r$ perturbation $E(\la)$ of arbitrarily small norm which is $\top$-even such that
the nonzero partial multiplicities at $0$ in $L+E$ are $(n_{r+1},\hdots,n_g)$ and such that all eigenvalues
different from those of $L$ are simple, as desired.

\cm{\emph{Subcase 1{\colb c}): $\lambda_0=0$ and $\top$-odd structure.}}
\cm{The case that $r$ is even can be treated analogously to the previous subcase 1{\colb b}), by just replacing
$1$ with $\la$ in the nonzero entries of the perturbation constructed above. However, the case {\colb when} $r$
is odd} deserves some more effort. The reason for this relies on the fact that any generic
$\top$-odd perturbation with rank $r$ and $r\leq n$ \cm{being odd contains $0$} as an eigenvalue.
This can be seen by looking at the summand \cm{$\la uu^\top$ in Theorem \ref{rank1-todd.th}}.
In this case, the part $L_0(\la)$ is a direct sum of blocks of two types:
\begin{itemize}
\item[(i)] A pair $2k\times 2k$ blocks of the form $R\diag (J_{2k}(\la),-J_{2k}(-\la))$.

\item[(ii)] A $(2k+1)\times(2k+1)$ block of the form
\[
U_{k}:=\left[\begin{array}{ccccccc}
&&&&&&\la\\
&&&&&\la&1\\
&&&&\iddots&\iddots\\
&&&\la&1\\
&&\la&-1\\
&\iddots&\iddots\\
\la&-1
\end{array}\right]_{(2k+1)\times (2k+1)}.
\]
\end{itemize}
\cm{Since the $\top$-odd perturbation pencil $E(\lambda)=\lambda E_A+E_B$ has odd rank $r$, it follows that the skew-symmetric
constant coefficient $E_{\colb B}$ has rank at most $r-1$. Then a straightforward dimension argument implies that the geometric multiplicity
of the eigenvalue zero can change at most by $r-1$. Hence, the geometric multiplicity of the eigenvalue zero must be at least $g-r+1$.
Since the list of partial multiplicities at zero must dominate the list $(n_{r+1},\dots,n_{g})$, but also {\colb must} contain, at least,
$g-r+1$ elements, the algebraic multiplicity of $n_{r+1}+\cdots+n_g$ is not possible for the eigenvalue zero. Now, the
(unique) list of partial multiplicities {\colb with minimal algebraic multiplicity} that dominates $(n_{r+1},\dots,n_{g})$ {\colb and} is consistent with a geometric multiplicity of, at least,
$g-r+1$ is the list $(n_{r+1},\dots,n_{g},1)$. Thus, by Theorem~\ref{local}, it remains to
construct one particular perturbation (of arbitrarily small norm) such that the perturbed pencil has this list of partial multiplicities
at zero and such that all eigenvalues different from those of the unperturbed pencil are simple to show that this is the generic case.}

Now, we are going to show how to construct \cm{such a} $\top$-odd perturbation, like in the previous case. For each pair of
blocks of type (i) we add the pencil $M_{k}:=(\la+\alpha)e_1e_{2k+1}^\top+(\la-\alpha)e_{2k+1}e_1^\top$, with
$e_1,e_{2k+1}\in\CC^{4k\times 4k}$. It is straightforward to see that
$\det(R\diag (J_{2k}(\la),-J_{2k}(-\la))+M_k)=(\la^{2k}-\la+\alpha)(\la^{2k}-\la-\alpha)$, and that the roots of this polynomial are simple
for $\alpha\neq0$. 

For each pair of blocks of type (ii), $U_{k_1}$ and $U_{k_2}$, we add a rank-$2$ perturbation of the form
$N_{k_1,k_2}:=\beta( e_1e_{2k_1+2}^\top- e_{2k_1+2}e_1^\top)$, with $e_1,e_{2k_1+2}\in\CC^{2(k_1+k_2+1)}$.
It is straightforward to see that $\det\left(\diag(U_{k_1},U_{k_2})+N_{k_1,k_2}\right)=(-1)^{k_1+k_2}\la^{2(k_1+k_2+1)}+\beta^2$,
so all the eigenvalues of the perturbed pencil are simple for $\beta\neq0$.

Finally, we must include a rank-$1$ summand of the form $\la uu^\top$ to get a perturbation like in \eqref{todd-rank1}.
This summand may correspond to either a pair of blocks of type (i) or to a block of type (ii) above.
The first case is not possible, since otherwise condition \eqref{property} would hold. Therefore, we must have a
block of the form $U_{\frac{n_r-1}{2}}$, and we add  a perturbation $\gamma e_1e_1^\top$, with $u_1\in\CC^{n_r}$.
It is straightforward to see that $\det (U_{\frac{n_r-1}{2}}+\gamma e_1e_1^\top)=(-1)^{\frac{n_r-1}{2}}\la^{n_r}+\la\gamma$.
Therefore, the perturbed pencil has $\la_0=0$ as a simple eigenvalue, and the remaining eigenvalues are simple for $\gamma\neq0$.

As before, choosing the parameters $\alpha,\beta$, and $\gamma$ appropriately,
we can construct a rank-$r$ perturbation $E(\la)$ of arbitrarily small norm which is $\top$-odd such that
the nonzero partial multiplicities at $0$ in $L+E$ are $(n_{r+1},\hdots,n_g,1)$ and such that all eigenvalues
different from those of $L$ are simple.

\cm{\emph{{\colb Subcase 1d}) $\lambda_0=\infty$}.}
\cm{For the eigenvalue $\la_0=\infty$ we just apply the result for $\la_0=0$ in the reversal pencil (recall
that $L(\la)$ is $\top$-even if and only if $\rev L(\la)$ is $\top$-odd).}

\cm{\emph{Case 2) Property~\eqref{property} applies}.}
Note that in this case we must have $\lambda_0=0$ or $\lambda_0=\infty.$
We distinguish several subcases.

\cm{\emph{{\colb Subc}ase 2a) $\la_0=0$ and $\top$-even structure}.
This case corresponds to} the first line of Table \ref{alter.table}. By part (1) of Theorem~\ref{local} we know
that, for any $\top$-even rank-$r$ pencil $E$, there are at least $g-r$ partial multiplicities
at $0$ in $L+E$, say $m_{r+1}\geq\cdots\geq m_g$, with $m_i\geq n_i$, for $i=r+1,\hdots,g$.
However, by the canonical form for $\top$-even pencils (see \cite[Th. 2.16]{batzke-thesis}
), it is not possible that these partial multiplicities be exactly $n_{r+1}\geq\cdots\geq n_g$,
because $L+E$ is $\top$-even, $n_{r+1}$ is odd, and its value appears an odd number of times in the
list $\{n_{r+1},\hdots,n_g\}$, by property \eqref{property}. Consequently, the algebraic multiplicity
$n_{r+1}+\cdots+n_g$ for the eigenvalue $\lambda_0$ of $L+E$ is not possible in this case.

\cm{As in the previous case}, we will instead construct a $\top$-even perturbation $E$ of rank $r$ and of arbitrarily
small norm such that the algebraic
multiplicity of $L+E$ at $0$ is $\widetilde a=n_{r+1}+\cdots+n_g+1$ and such that all eigenvalues that are different
from those of $L$ are simple. Then by part (2) of Theorem~\ref{local}
there is a generic set $\cG\subseteq\CC^{\lfloor\frac{3r}{2}\rfloor n}$ such that for all corresponding perturbations $E$
we have the situation outlined above.

As before, let us assume that $L(\la)$ is given in $\top$-even canonical form, so we can write it as
\[
L(\la)=\diag\big(L_1(\la),R\diag (J_{n_r}(-\la),J_{n_r}(\la)),L_2(\la) ,\widetilde J(\la)\big),
\]
where $L_1(\la)$ contains the first $r-1$ Jordan blocks associated with $0$, $ L_2(\la)$
contains the Jordan blocks associated with $0$ and with sizes $n_{r+2},\hdots,n_g$, and
$\widetilde J(\la)$ corresponds to the nonzero eigenvalues (including infinity). Here, we used the fact that
$n_r=n_{r+1}$ by property~\eqref{property}. Now, let $E(\la)$ be of the form
\[
E(\la)=\diag (E_1,\gamma(e_1+e_{n_r+2})(e_1+e_{n_r+2})^\top, 0)
\]
where $e_1,e_{n_r+2}\in\CC^{2n_r}$ (with $e_{n_r+2}$ interpreted as being the zero vector in the case $n_r=1$),
and where $E_1$ is of size
$(n_1+\cdots+n_{r-1})\times (n_1+\cdots+n_{r-1})$ and is constructed as a direct sum of
blocks as explained above for the precedent case associated with the eigenvalue $\la_0=0$.
(Namely, $E_1$ consists of a direct sum of rank-$1$ blocks with sizes $n_i\times n_i$ or
rank-$2$ blocks with sizes $(n_i+n_{i+1})\times (n_i+n_{i+1})$, depending on whether $L_1(\la)$
contains a $n_i\times n_i$ block, with $n_i$ even, or a pair of blocks with sizes
$n_i\times n_i$ and $n_{i+1}\times n_{i+1}$, with $n_{i+1}=n_i$ odd.) Then
\begin{equation}\label{detpert}
\begin{array}{ccl}
\det(L+E)&=&\det (L_1(\la)+E_1)\\
&&\cdot\det\big(R\diag (J_{n_r}(-\la),J_{n_r}(\la))+\gamma(e_1+e_{n_r+2})(e_1+e_{n_r+2})^\top)\big)\\
&&\cdot\det L_2(\la)\cdot \det \widetilde J(\la).
\end{array}
\end{equation}
With straightforward computations (using again the Laplace expansion) \cm{it can be seen that%
\begin{equation}\label{eq:11.1.19}
\det(R\diag (J_{n_r}(-\la),J_{n_r}(\la))+\gamma(e_1+e_{n_r+2})(e_1+e_{n_r+2})^\top))=
\la^{n_r+1}(\lambda^{n_r-1}-\cm{2\gamma})
\end{equation}
if $n_r>1$, or $\det(R\diag (J_{n_r}(-\la),J_{n_r}(\la))+\gamma(e_1+e_{n_r+2})(e_1+e_{n_r+2})^\top))=\lambda^2$ if $n_r=1$,
see Appendix~\ref{appendix} (see also \cite[p. 663]{batzke14}).}
On the other hand, we have $\det L_2(\la)=\lambda^{n_{r+2}+\cdots+n_g}$. Thus, choosing the parameters
$\alpha_i$ and $\beta_j$ in $E_1$ and the parameter $\gamma$ appropriately, we can construct a perturbation pencil $E$
of arbitrarily small norm such that the algebraic multiplicity of $L+E$ at zero is $\widetilde a=n_{r+1}+\cdots+n_g+1$
and such that all eigenvalues of $L+E$ that are different from those of $L$ are simple, as desired.

However, the reader should keep in mind that part (2) of Theorem~\ref{local} only contains information on the
generic algebraic multiplicity of the eigenvalue $0$ of $L+E$ for a generic $\top$-even perturbation $E$.
Unlike the previous cases, it is no longer true that combining the parts (1) and (2) of Theorem~\ref{local}
forces the partial multiplicities of $L+E$ at $0$ to be uniquely determined. Therefore, it is necessary to further
investigate which lists of partial multiplicities at $0$ are possible such that both (1) and (2) of Theorem~\ref{local}
are satisfied. To this end, there are three possible situations:
\begin{itemize}
\item[(a)] If $\cm{n_{r+1}}-1\not\in\{ n_{r+2},\hdots,n_g,0\}$, then the only possible partial multiplicities
are $n_{r+1}+1> n_{r+2}\geq\cdots\geq n_g$.

\item[(b)] If $\cm{n_{r+1}}-1\in\{n_{r+2},\hdots,n_g\}$, say $\cm{n_{r+1}}-1=n_{r+d+1}$ (and $ d$ being minimal with this property),
then there are two possible lists of partial multiplicities:
\begin{itemize}
\item[(b1)] $n_{r+1}+1> n_{r+2}\geq\cdots\geq n_g$, or
\item[(b2)] $n_{r+1}=\cdots=n_{r+ d}=n_{r+ d+1}+1> n_{r+d+2}\geq\cdots\geq n_g$.
\end{itemize}
\item[(c)] If $n_r=1$, then there are two possible lists of partial multiplicities:
\begin{itemize}
\item[(c1)] $(2,\underbrace{1,\hdots,1}_{g-r-1})$, or
\item[(c2)] $(\underbrace{1,\hdots,1}_{\cm{g-r+1}})$.
\end{itemize}
\end{itemize}
To see this, first note that, for any $x\in \cG$, the algebraic multiplicity of $L+\Phi(x)$ at $0$ is,
exactly, $\widetilde a$. Since the partial multiplicities at $0$ are $m_{r+1}\geq \cdots\geq m_g$,
with $m_i\geq n_i$, for $i=r+1,\hdots,g$, then either one of the partial multiplicities
$n_{r+1}\geq\cdots\geq n_g$ at $0$ in $L$ increases one unit, or either a new partial multiplicity
equal to $1$ appears after adding $E=\Phi(x)$. However, it is not possible to add or remove
just one odd partial multiplicity after perturbing by $E$, since this would imply that the
parity in the number of some of the odd-sized Jordan blocks associated with $0$ would change,
and this is not allowed by the $\top$-even structure. However, when increasing in one unit
just one partial multiplicity at $0$ in $L$, say $n_i$, either one odd partial multiplicity
is added or removed, depending on the parity of $n_i$. In order for the number of each odd-sized
Jordan blocks associated with $0$ to stay as an even number, the only possibility is that
either $n_i=\cm{n_{r+1}}$ or $n_i=\cm{n_{r+1}}-1$. The first case corresponds to cases (a), (b1), and (c1) above,
whereas the second one corresponds to cases (b2) and (c2).

With an argument identical to the one used in \cite{batzke-thesis},
we are going to prove that the generic partial multiplicities are just the ones in either (a),
(b1), or (c1), which essentially reduce to the same behavior, namely, one of the largest remaining partial
multiplicities increases in one unit.

Let us focus on case (b) first. By assumption on $d$ being minimal,
we have $(n_r=)n_{r+1}=\cdots=n_{r+ d}>n_{r+ d+1}\geq\cdots\geq n_g$ and $\cm{n_{r+1}}-1=n_{r+ d+1}$. Note that
necessarily $ d$ is odd as we are in the case of property~\eqref{property}.

Assume that the
change in case (b1) is not generic. Then the set $\mathcal B\subseteq\CC^{\lfloor\frac{3r}{2}\rfloor n}$
of all $x$ for which the partial multiplicities of $L+\Phi(x)$
at $0$ are $n_{r+1}=\cdots=n_{r+ d}=n_{r+d+1}+1> n_{r+ d+2}\geq\cdots\geq n_g$ is not contained in a proper
algebraic set. (Note that it must happen that $g-r\geq2$.)

Now, let us define the map
\[
\begin{array}{cccl}
\vm{\widetilde \Phi_d}:&(\CC^n)^{ d}&\longrightarrow&Even^\top_{d}\\
&u=(u_1,\dots,u_{d})&\mapsto&\widetilde\Phi(u)=u_1u_1^\top+\cdots+u_{ d}u_{d}^\top.
\end{array}
\]
and also consider the map
\[
\begin{array}{cccl}
\widetilde\Phi:&\CC^{\lfloor\frac{3r}{2}\rfloor n}\times(\CC^n)^{ d}&\longrightarrow &Even_{r+{ d}}^\top\\
&(x,u)&\mapsto&\widetilde \Phi(x,u)=\Phi(x)+\widetilde\Phi_{d}(u),
\end{array}
\]
Observe that the map $\widetilde\Phi$ may be different from the corresponding map
$\CC^{\lfloor\frac{3(r+{\colb d})}{2}\rfloor n}\longrightarrow Even_{r+ d}^\top$ from~\eqref{phione}.
(Indeed, the dimensions of the domains do not coincide if $r$ is odd.) Moreover, it is not
even clear whether the map $\widetilde \Phi$ is surjective.
Nevertheless, $\widetilde\Phi$ satisfies the hypotheses of Theorem~\ref{local} and thus by
part (1) of Theorem~\ref{local} we have that for any $(x,u)\in\mathcal B\times(\CC^n)^{ d}$
the list of partial multiplicities of $L+\widetilde\Phi(x,u)$ at $\lambda_i$ dominates the list
$n_{r+ d+1}+1> n_{r+ d+2}\geq\cdots\geq n_g$. The key observation is now that by \cite[Lemma 2.2]{batzke16} the set
$\mathcal B \times(\CC^n)^{ d}$ is not contained in a proper algebraic subset of
$\CC^{\lfloor\frac{3r}{2}\rfloor n}\times(\CC^n)^{d}$. If we can show that there exist $(x_0,u_0)$ of arbitrarily
small norm such that the partial multiplicities of $L+\widetilde\Phi(x_0,u_0)$ are $n_{r+d+1}\geq n_{r+ d+2}\geq\cdots\geq n_g$,
then by part (2) of Theorem~\ref{local} this hold for all $L+\widetilde\Phi(x,u)$ with $(x,u)$ from a generic
set $\widetilde\cG\subseteq\CC^{\lfloor\frac{3r}{2}\rfloor n}\times(\CC^n)^{d}$. Since the list
$n_{r+d+1}\geq n_{r+d+2}\geq\cdots\geq n_g$ does not dominate the list $n_{r+ d+1}+1> n_{r+ d+2}\geq\cdots\geq n_g$
this leads to a contradiction, because the sets $\widetilde\cG$ and $\mathcal B\times (\CC^n)^{ d}$ must have a
nonempty intersection, the first set being generic and the second set not being contained in a proper algebraic set.

Thus it remains to construct one particular example with the properties outlined above. To this end, note that,
by assumption on $k$, the pencil $L$ has the form
\[
L(\la)=\diag\big(L_1(\la),R\diag (J_{n_r}(-\la),J_{n_r}(\la)),\dots, R\diag (J_{n_r}(-\la),J_{n_r}(\la)),L_3(\la) ,\widetilde J(\la)\big),
\]
where the block $R\diag (J_{n_r}(-\la),J_{n_r}(\la))$ is repeated $\frac{ d+1}{2}$ times and $L_3(\lambda)$ contains
the blocks associated with the partial multiplicities $n_{r+k+2}\geq\cdots\geq n_g$. Then the desired example for a
perturbation that does the job is given by
\[
E(\la)=\gamma\diag (E_1,e_1e_1^\top+ e_{n_r+1}e_{n_r+1}^\top,\dots,e_1e_1^\top+ e_{n_r+1}e_{n_r+1}^\top, 0),
\]
where $E_1$ is as before, the block $e_1e_1^\top+ e_{n_r+1}e_{n_r+1}^\top$ is repeated $\frac{ d+1}{2}$ times, and $\gamma>0$
is chosen sufficiently small. Indeed note that, as before, all blocks in $L_1$ and all the paired blocks of size $n_r$ are
perturbed in such a way that all eigenvalues lie on circles around zero, so that the partial multiplicities of $L+E$ at $0$
are given by $n_{r+d+2}\geq\cdots\geq n_g$. Moreover, $E_1+e_1e_1^\top$ is a $\top$-even pencil of rank $r$ and thus, using
the surjectivity of $\Phi$, there exists $x\in\CC^{\lfloor\frac{3r}{2}\rfloor n}$ with $\Phi(x)=E_1+e_1e_1^\top$. Since
the remaining part of $E$ is of the form $u_1u_1^\top+\cdots+u_{d}u_{d}^\top$, this implies the existence of
$(x,u)\in\CC^{\lfloor\frac{3r}{2}\rfloor n}\times(\CC^n)^{ d}$ with $\widetilde\Phi(x,u)=E$.

To show that in case (c) the subcase (c1) is generic can be shown by contradiction in a similar way. In this case, there would be
two generic sets of $\top$-even perturbations with rank $r+1$ giving different behavior.

\cm{\emph{Subcase {\colb 2}b) $\la_0=0$ and $r$ odd and $\top$-odd structure}.}
\cm{In this case, the situation is similar to the one in the previous subcase, but we are also in a situation similar
to the one in Subcase {\colb1c}), i.e., the geometric multiplicity of the eigenvalue $\la_0=0$ after perturbation must be at least $g-r+1$.
But then, it is straightforward to show that the algebraic multiplicity $n_{r+1}+\cdots+n_g+1$ is not possible in this case.
Thus, we will construct a perturbation leading to the algebraic multiplicity $\widetilde a=n_{r+1}+\cdots+n_g+2$.
As before, let us assume that $L(\la)$ is given in $\top$-odd canonical form, so we can write it as
\[
L(\la)=\diag\big(L_1(\la),R\diag (J_{n_r}(-\la),J_{n_r}(\la)),L_2(\la) ,\widetilde J(\la)\big),
\]
where $L_1(\la)$ contains the first $r-1$ Jordan blocks associated with $0$, $ L_2(\la)$
contains the Jordan blocks associated with $0$ and with sizes $n_{r+2},\hdots,n_g$, and
$\widetilde J(\la)$ corresponds to the nonzero eigenvalues (including infinity). Since the pencil $L_1(\lambda)$ does not have
the property~\eqref{property}, we can construct a $\top$-odd perturbation $E_1(\lambda)$ as in \vm{subcase 1b) } such that the
eigenvalues of the perturbed pencil $L_1+E_1$ are all nonzero and simple. It remains to perturb the block
$R\diag (J_{n_r}(-\la),J_{n_r}(\la))$ in an appropriate way. For this} we consider the
perturbation $\gamma \lambda(e_1+e_{n_r+2})(e_1+e_{n_r+2})^\top$, with $e_1,e_{n_r+2}\in\CC^{2n_r}$. It is straightforward to see that
\begin{equation}\label{det-identity}
\det(R\diag (\cm{-J_{n_r}(-\la),J_{n_r}(\la)})+\gamma \lambda(e_1+e_{n_r+2})(e_1+e_{n_r+2})^\top=\la^{n_r+2}\left(\la^{n_r-2}+2\gamma\right)
\end{equation}
(a proof of this identity is provided in Appendix \ref{appendix}).
%
%
Moreover, since, for $\la_0=0$, the perturbed \cm{subpencil} has the same rank as the original one, the geometric multiplicity
of $\la_0=0$ at the perturbed \cm{subpencil} is the same one as in the original one, namely $2$.
Therefore, \cm{setting $E(\lambda)=\diag(E_1,\gamma \lambda(e_1+e_{n_r+2})(e_1+e_{n_r+2})^\top,0)$ and choosing
$\gamma$ sufficiently small}, the eigenvalues of the perturbed pencil $L+E$
are $\la_0=0$ with algebraic multiplicity \cm{$\widetilde a$, geometric multiplicity $g-r+1$}, and the remaining eigenvalues are all simple,
for $\gamma\neq0$. \cm{The argument that the geometric multiplicities are as claimed in Table~\ref{alter.table} is shown in a way that
is analogous to the one in Subcase {\colb1c}).}

\cm{\emph{Subcase 2c) $\lambda_0=\infty$}.}
The cases $\la_0=\infty$ where property \eqref{property} appears in Table \ref{alter.table} can be proved from the cases
$\la_0=0$ by using the reversal, which exchanges the roles of these two eigenvalues and takes $\top$-even pencils into
$\top$-odd ones and viceversa. In particular, the case $\la_0=\infty$ in the $\top$-odd structure can be obtained
from the case $\la_0=0$ in the $\top$-even structure, and the case $\la_0=\infty$ in the $\top$-even case can
be obtained from the case $\la_0=0$ in the $\top$-odd structure.
\end{proof}
\begin{theorem}\label{tpal-main.th} {\rm (Generic change under low-rank perturbations of
$\top$-palindromic pencils).}
Let $\la_1,\dots,\lambda_\kappa$ be the pairwise distinct eigenvalues of the regular
$n\times n$ $\top$-palindromic or $\top$-anti-palindromic
matrix pencil $L(\la)$, having the nonzero partial multiplicities $n_{i,1}\geq n_{i,2}\geq\cdots\geq n_{i,g_i}>0$,
for $k=1,\dots,\kappa$, respectively. Furthermore, let $r>0$ be an integer and let $\Phi$ be the map as in Remark {\rm\ref{param.rem}}.
Then, there is a generic set $\cG$ in $\CC^{\lfloor\frac{3r}{2}\rfloor n}$ such that, for
all $E(\la)\in\Phi(\cG)$, the perturbed pencil $L+E$ is regular and the partial multiplicities
of $L+E$ at $\la_0$ are the ones given in Table {\rm\ref{tpal.table}}, where \eqref{property}
is the same property as in the statement of Theorem {\rm\ref{talter-main.th}}.
(In particular, if $r\geq g_i$ then $\lambda_i$ is not an eigenvalue of $L+E$.) Furthermore,
all eigenvalues of $L+E$ different from those of $L$ are simple.
\begin{table}[htb!]\footnotesize
\begin{tabular}{|c|c|l|c|}\hline
Structure&e-val $\la_i$&case&multiplicities\\\hline\hline
\multirow{7}{*}{$\top$-palindromic}&
\multirow{2}{*}{$\la_i=1$}&$n_{i,r+1}$ odd and \eqref{property} holds&$\cm{(n_{i,r+1}+1,n_{i,r+2},\dots,n_{i,g_i})}$\\
&& otherwise&$\cm{(n_{i,r+1},n_{i,r+2},\dots,n_{i,g_i})}$\\\cline{2-4}    
&\multirow{4}{*}{$\la_i=-1$}& $r$ even, $n_{i,r+1}$ even, and \eqref{property} holds&$\cm{(n_{i,r+1}+1,n_{i,r+2},\dots,n_{i,g_i})}$\\
&& $r$ even, otherwise&$\cm{(n_{i,r+1},n_{i,r+2},\dots,n_{i,g_i})}$\\
&&$r$ odd, $n_{i,r+1}$ even, and \eqref{property} holds&$\cm{(n_{i,r+1}+1,n_{i,r+2},\dots,n_{i,g_i},1)}$\\
&& $r$ odd, otherwise&$\cm{(n_{i,r+1},n_{i,r+2},\dots,n_{i,g_i},1)}$\\\cline{2-4}     
&$\la_i\in\CC\setminus\{\pm 1\}$&all&$\cm{(n_{i,r+1},n_{i,r+2},\dots,n_{i,g_i})}$\\\hline\hline
\multirow{7}{*}{$\top$-anti-palindromic}&
\multirow{4}{*}{$\la_i=1$}& $r$ even, $n_{i,r+1}$ even, and \eqref{property} holds&$\cm{(n_{i,r+1}+1,n_{i,r+2},\dots,n_{i,g_i})}$\\
&&  $r$ even, otherwise&$\cm{(n_{i,r+1},n_{i,r+2},\dots,n_{i,g_i})}$\\
&& $r$ odd, $n_{i,r+1}$ even, and \eqref{property} holds&$\cm{(n_{i,r+1}+1,n_{i,r+2},\dots,n_{i,g_i},1)}$\\
&&  $r$ odd, otherwise&$\cm{(n_{i,r+1},n_{i,r+2},\dots,n_{i,g_i},1)}$\\\cline{2-4}     
&\multirow{2}{*}{$\la_i=-1$}&$n_{i,r+1}$ odd and \eqref{property} holds&$\cm{(n_{i,r+1}+1,n_{i,r+2},\dots,n_{i,g_i})}$\\
&& otherwise&$\cm{(n_{i,r+1},n_{i,r+2},\dots,n_{i,g_i})}$\\\cline{2-4}     
&$\la_i\in\CC\setminus\{\pm 1\}$&all&$\cm{(n_{i,r+1},n_{i,r+2},\dots,n_{i,g_i})}$\\\hline
\end{tabular}
\caption{\footnotesize Generic partial multiplicities at $\la_0$ for rank-$r$ $\top$-alternating perturbations}\label{tpal.table}
\end{table}
\end{theorem}
\begin{proof}
We just prove the $\top$-palindromic case, since the $\top$-anti-palindromic one follows similar reasonings.

Let $L(\la)$ be a given $\top$-palindromic pencil satisfying the conditions in the statement,
and let $E(\la)$ be another $\top$-palindromic pencil of the form \eqref{tpal-rank1}.
Let ${\cal C}_{+1}$ and ${\cal C}_{-1}$ be the Cayley transforms in \eqref{cayley}. Then
\[
{\cal C}_{+1}(L+E)={\cal C}_{+1}(L)+{\cal C}_{+1}(E),
\]
with both the pencil in the left-hand side and the ones in the right-hand side being $\top$-even
\cite[Th. 2.7]{4m-good}. Moreover, if $r=\rank E$ is odd, then
\begin{equation}\label{cayley-tpalrank1}
\begin{array}{ccl}
{\cal C}_{+1}(E)(\mu)&=&{\cal C}_{+1}((1+\la)uu^\top+v_1w_1^\top+\cdots+v_{(r-1)/2}w_{(r-1)/2}^\top\\
&&+(\rev w_1)v_1^\top+\cdots+(\rev w_{(r-1)/2})v_{(r-1)/2}^\top)\\
&=&2uu^\top+v_1\widehat w_1(\mu)^\top+\cdots+v_{(r-1)/2}\widehat w_{(r-1)/2(\mu)}^\top\\
&&+(\widehat w_1(-\mu))v_1^\top+\cdots+ ( \widehat w_{(r-1)/2}(-\mu))v_{(r-1)/2}^\top),
\end{array}
\end{equation}
with $\widehat w_i(\mu)={\cal C}_{+1}(w_i)(\mu)=(1-\mu)w(\frac{1+\mu}{1-\mu})$, for $i=1,\hdots,(r-1)/2$.
The second sum in the last term of \eqref{cayley-tpalrank1} follows by using similar identities to the
ones in \eqref{reversal-w}, which allow us to see that
\[
\begin{array}{ccl}
{\cal C}_{+1}(\rev w_i)(\mu)&=&{\cal C}_{+1}(\la w_i(1/\la))(\mu)=(1-\mu)\cdot\frac{1+\mu}{1-\mu}\cdot
w_i\left(\frac{1-\mu}{1+\mu}\right)\\
&&=(1+\mu)w_i\left(\frac{1-\mu}{1+\mu}\right)=\widehat w_i(-\mu).
\end{array}
\]
If $r$ is even, then we get a similar expression according to the expression for $E(\la)$ in \eqref{tpal-rank1}.
This means that the pencil ${\cal C}_{+1}(E)$ is of the form \eqref{teven-rank1}. Then, by
Theorem \ref{talter-main.th}, there is a generic set $\cG$ in $\CC^{\lfloor\frac{3r}{2}\rfloor n}$ such that,
for all $x\in \cG$, the perturbed pencil $({\cal C}_{+1}(L)+\Phi(x))(\mu)$ is regular and
the partial multiplicities at $\mu_0$ are the ones given in Table {\rm\ref{alter.table}},
replacing $\mu_0$ by $\la_i$, with $\mu_0=(\la_i-1)/(\la_i+1)$ if $\la_i\neq 1$, and $\mu_0=\infty$
if $\la_i=1$, and furthermore, such that all eigenvalues that are different from those of ${\cal C}_{+1}(L)$ are simple.
Note that $\Phi$ is the map that takes a set of parameters
$x\in\CC^{\lfloor\frac{3r}{2}\rfloor n}$ to a pencil like in \eqref{cayley-tpalrank1}.

Applying the Cayley transformation ${\cal C}_{-1}$ we conclude that, for any $x\in \cG$,
the pencil $L+{\cal C}_{-1}(\Phi(x))$ is regular and has the partial multiplicities at $\la_i$ as
given in Table \ref{tpal.table}, while all eigenvalues that are different from those of $L$ are simple.
But, since $\Phi(x)={\cal C}_{+1}(E)(\mu)$, then
${\cal C}_{-1}(\Phi(x))=E(\la)$, and this concludes the proof for this case.

For the $\top$-anti-palindromic case just replace ${\cal C}_{-1}$ by ${\cal C}_{+1}$ and
vice versa, and refer to the $\top$-odd case instead of the $\top$-even one.
\end{proof}

Next, we turn to the skew-symmetric structure. As it is well known, the algebraic multiplicity of each
eigenvalue of a skew-symmetric pencil is necessarily even (see, e.g., \cite[Theorem 2.18]{batzke-thesis}).
As a consequence, the newly generated eigenvalues by a structure-preserving perturbation will generically
be double eigenvalues instead of simple ones.
\begin{theorem}\label{skews-main.th}{\rm (Generic change under low-rank perturbations of skew-symmetric pencils).}
Let $L(\lambda)$ be a regular $n\times n$ skew-symmetric matrix pencil and let $\lambda_1,\dots,\lambda_\kappa\in\mathbb C$
be its pairwise distinct eigenvalues having the nonzero partial multiplicities
$n_{i,1}\geq n_{i,2}\geq\cdots\geq n_{i,g_i}>0$, for $i=1,\dots,\kappa$, respectively.
(We highlight that both $n$ and all values $n_{i,j}$, $j=1,\dots,g_i$, $i=1,\dots,\kappa$ are necessarily even.)
Furthermore, let $r$ be a nonzero even integer and let $\Phi$ be the map as in Remark~{\rm\ref{param.rem}}.
Then, there is a generic set $\cG$ in $\CC^{\frac{3rn}{2}}$ such that, for all $E(\la)\in\Phi(\cG)$, the perturbed
pencil $L+E$ is regular and the partial multiplicities of $L+E$ at $\la_i$ are $n_{i,r+1}\geq\cdots\geq n_{i,g_i}$, for $i=1,\dots,\kappa$.
Furthermore, all eigenvalues of $L+E$ that are \vm{not} eigenvalues of $L$ have algebraic multiplicity precisely two.
\end{theorem}
\begin{proof}
Without loss of generality we may assume that $L(\lambda)$ is of the form
\[
L(\lambda)=\mat{cc}0&D(\lambda)\\ -D(\lambda)&0\rix,
\]
where $D(\lambda)$ is a regular pencil of size $\frac{n}{2}\times\frac{n}{2}$. This assumption can be made since $L(\lambda)$ is
congruent to a pencil in the indicated form - a fact that follows easily by assuming that
$L(\lambda)$ is in the canonical form of \cite[Theorem 2.18]{batzke-thesis} and then applying simultaneous row and
column permutations. Clearly, the eigenvalue $\lambda_i$ of $D(\lambda)$ has the partial multiplicities
$\frac{n_{i,1}}{2}\geq \frac{n_{i,2}}{2}\geq\cdots\geq \frac{n_{i,g_i}}{2}$. By the proof of Theorem~\ref{gen-main.th},
there exists $\widetilde x\in\mathbb C^{\frac{3rn}{4}}$ of arbitrarily small norm such that
$\widetilde E(\lambda)=\Phi_{\frac{r}{2}}(\widetilde x)$ (with $\Phi_{\frac{r}{2}}$ being the map from Definition~\ref{param.def.gen})
is an $\frac{n}{2}\times\frac{n}{2}$ pencil of rank $\frac{r}{2}$ such that $D+\widetilde E$ is regular, has the partial multiplicities
$\frac{n_{i,r+1}}{2}\geq\cdots\geq \frac{n_{i,g_i}}{2}$ at $\lambda_i$, for $i=1,\dots,\kappa$, and all its eigenvalues
that are different from those of $D$ are simple. Then setting
\[
E(\lambda)=\mat{cc}0&\widetilde E(\lambda)\\ -\widetilde E(\lambda)&0\rix,
\]
it follows that $E$ is skew-symmetric and has rank $r$. Furthermore, due to the surjectivity of $\Phi$ it follows
that there exists $x\in\CC^{\frac{3rn}{2}}$ such that $\Phi(x)=E$ and it is straightforward to check that
$x$ can be chosen to be of the same norm as $\widetilde x$. Obviously, $L+E$ now has the partial multiplicities
$n_{i,r+1}\geq\cdots\geq n_{i,g_i}$ at $\lambda_i$ for $i=1,\dots,\kappa$,
and all eigenvalues of $L+E$ that are no{\colb t} eigenvalues of $L$ have algebraic multiplicity precisely two.
Then applying Theorem~\ref{local} with $\mu=2$ yields the desired result.
\end{proof}

As for the remaining structures (skew-Hermitian, $*$-alternating, $*$-palindromic, and
$*$-anti-palindromic) a similar result to Theorem \ref{herm-main.th} can be obtained
either from this result directly using the observations in the paragraph right after
Theorem \ref{rank1-tantipal.th} (skew-Hermitian, $*$-alternating) or using appropriate Cayley
transformations as in the proof of Theorem \ref{tpal-main.th} ($*$-palindromic,
and $*$-anti-palindromic). We gather all these results in just one statement in Theorem \ref{other-main.th}.
\begin{theorem}\label{other-main.th} {\rm (Generic change under low-rank perturbations
of skew-Hermitian, $*$-alternating, $*$-palindromic, and $*$-anti-palindromic pencils).}
Let $\la_1,\dots,\lambda_\kappa$ be the pairwise distinct eigenvalues (finite or infinite) of the regular $n\times n$ skew-Hermitian,
$*$-alternating, $*$-palindromic, or $*$-anti-palindromic matrix pencil $L(\la)$, with
nonzero partial multiplicities $n_{i,1}\geq n_{i,2}\geq\cdots\geq n_{i,g_i}>0$ for $i=1,\dots,\kappa$, respectively.
Furthermore, let $r$ be a positive integer and,
for each $0\leq s\leq\lfloor r/2\rfloor$, let $\Phi_s$ be the map as in Remark
{\rm\ref{param.rem}}. Then, there is a generic set $\cm{\cG_s}$ in $\cm{\mathbb R^\ell\times}\CC^{(r+s)n}$ such that,
for all $E(\la)\in\Phi_s(\cm{\cG_s})$, the perturbed pencil $(L+E)(\la)$ is regular and the
partial multiplicities of $L+E$ at $\la_i$ are $n_{i,r+1}\geq\cdots\geq n_{i,g_i}$ for $i=1,\dots,\kappa$. In particular,
if $g_i\leq r$ then $\lambda_i$ is not an eigenvalue of $L+E$. Furthermore, all eigenvalues of $L+E$
that are not eigenvalues of $L$ are simple.
\end{theorem}
%
 The results presented in Theorems \ref{sym-main.th}--\ref{tpal-main.th} extend the ones in \cite{batzke14}
\cm{and \cite{batzke16}} from rank-$1$ \cm{and special rank-$2$ perturbations} to low-rank perturbations of matrix pencils
\cm{with symmetry structures}. Even though some of the arguments and techniques in the proof of
Theorems \ref{sym-main.th}--\ref{tpal-main.th} are analogous to some of the ones used in \cm{\cite{batzke14,batzke16}}, the main approach,
which uses the parameterizations constructed from the rank-$1$ decompositions given in Section \ref{rank1.sec}, is different
to the one followed in \cm{\cite{batzke14,batzke16}}.

If we compare Theorems \ref{herm-main.th}--\ref{other-main.th} with Theorem \ref{gen-main.th}, we will realize that, in most cases,
the generic behavior for \cm{pencils with symmetry structures} coincides with the one for \cm{general} pencils. However,
there are several cases in Theorems \ref{talter-main.th} and \ref{tpal-main.th} where this behavior is different.
In these cases, the \cm{$\top$-alternating and $\top$-palindromic structures} impose additional restrictions that must be
fulfilled in the canonical form, which prevent some behaviors, that in the \cm{general} case are allowed, to occur under
\cm{structure-preserving} perturbations of pencils having these symmetry structures.

\section{Outlook on the real case}\label{real.sec}

So far, we have restricted ourselves to the complex case only. The main reason for this is the
surprising fact that in general real versions of rank-$1$ decompositions as in Theorem~\ref{rank1-herm.th}
or Theorem~\ref{rank1-sym.th} need not exist as the following example shows.
\begin{example}\rm
Consider the real symmetric pencil
\[
E(\lambda)=2\mat{cc}0&1\\ 1&0\rix-\lambda\mat{cc}1&0\\ 0&-1\rix=\mat{cc}-2\lambda&2\\ 2&2\lambda\rix.
\]
This pencil has the eigenvalues $\iunit,-\iunit$ and a decomposition in complex Hermitian rank-$1$ pencils is
given by
\[
E(\lambda)=\mat{c}1\\ -\iunit\rix\mat{cc}-\lambda&2+\iunit\lambda\rix+
\mat{c}-\lambda\\ 2-\iunit\lambda\rix\mat{cc}1&\iunit\rix
\]
while for a decomposition into complex symmetric rank-$1$ pencil pencils we can take
\[
E(\lambda)=(\lambda+\iunit)\mat{c}-\iunit\\ 1\rix\mat{cc}-\iunit&1\rix+(\lambda-\iunit)\mat{c}\iunit\\ 1\rix\mat{cc}\iunit&1\rix.
\]
However, $E(\lambda)$ does not allow a decomposition of the form
\begin{equation}\label{eq:realnotpossible}
E(\lambda)=v(w+\lambda x)^\top+(w+\lambda x)v^\top\quad\mbox{with }\;
v=\mat{c}v_1\\ v_2\rix,\;w=\mat{c}w_1\\ w_2\rix,\;x=\mat{c}x_1\\ x_2\rix\in\mathbb R^2.
\end{equation}
Indeed,~\eqref{eq:realnotpossible} leads to the contradictory equations
\[
\mat{cc}2v_1w_1&v_1w_2+v_2w_1\\ v_1w_2+v_2w_1&2v_2w_2\rix=\mat{cc}0&2\\ 2&0\rix\quad\mbox{and}\quad
\mat{cc}2v_1x_1&v_1x_2+v_2x_1\\ v_1x_2+v_2x_1&2v_2x_2\rix=\mat{cc}-2&0\\ 0&2\rix,
\]
since these imply $v_1,v_2\neq 0$ and thus $w_1=w_2=0$, which contradicts $v_1w_2+v_2w_1=2$.
But $E(\lambda)$ does not allow a decomposition of the form
\[
E(\lambda)=(a_1+\lambda b_1)uu^\top +(a_2+\lambda b_2)vv^\top\quad\mbox{with }\;
u=\mat{c}u_1\\ u_2\rix,\;v=\mat{c}v_1\\ v_2\rix
\]
either, because in that case the pencil would have real eigenvalues, which is not the case.
\end{example}

We expect that in the real case one will have to allow summands of rank two in order to obtain
a decomposition into low-rank pencils. This will be subject to subsequent research.

\section{Conclusions and future work}\label{conclusion.sec}
We have described the generic change of the Weierstra\ss\ Canonical Form (given by the partial multiplicities) of regular
\cm{matrix pencils with symmetry structures} under \cm{structure-preserving} additive low-rank perturbations. In particular,
we have considered all the
structures indicated at the beginning of Section \ref{rank1.sec}.  We have seen that, for most eigenvalues and most of the structures,
the generic change coincides with the one in the unstructured case, namely: given an eigenvalue $\la_0\in\CC\cup\{\infty\}$ of
the pencil $L(\la)$, with $g$ associated partial multiplicities, for a generic perturbation, $E(\la)$, of rank $r$, the partial
multiplicities of $(L+E)(\la)$ at $\la_0$ are exactly the $g-r$ smallest partial multiplicities of $L(\la)$. In particular,
if $r\geq g$, the value $\la_0$ is not generically an eigenvalue of $(L+E)(\la)$. However, for the $\top$-alternating structures,
there is a (generic) different behavior for the eigenvalues $\la_0=0$ and $\la_0=\infty$, and similarly for the $\top$-palindromic
structures with the eigenvalues $\la_0=\pm1$. These differences arise in those cases where the parity of the partial multiplicities
in the perturbed pencil $L+E$ provided by the generic behavior in the unstructured case is not in accordance with the restrictions
imposed by the structure (for instance, the even-sized blocks associated with $\la_0=0$ in $\top$-even pencils must be paired-up).

Our results contain the ones in \cite{batzke14}, valid for rank-$1$ perturbations of \cm{pencils with symmetry structure, and extend the
ones in \cite{batzke16} that are valid for special rank-$2$ perturbations of pencils with symmetry structures. However, }
the main tools and developments used in this work are different to the ones in \cite{batzke14,batzke16}. More precisely, to obtain our
main results we have introduced a \cm{structure-preserving} rank-$1$ decomposition of low-rank pencils \cm{with symmetry structures},
for each of the structures  considered in the paper.

Several lines of research arise as a natural continuation of this work:

\begin{itemize}
\item To analyze the generic change of the partial multiplicities under low-rank perturbations of \cm{pencils with symmetry structures
that have real coefficients}, together with the generic change of the sign characteristic. In this work, we have restricted ourselves to the
partial multiplicities, but the sign characteristic is also a key ingredient in the eigenstructure, for instance, of Hermitian pencils.
The sign characteristic also appears in \cm{matrix pencils with real coefficients}, for some of the other structures
considered in this work (like the $\top$-even structure, see \cite{thompson}). So it is natural to address the generic
change of the sign characteristic in the context of \cm{pencils with symmetry structures having} real coefficients.
\item To describe the generic change of the partial multiplicities under low-rank perturbations of regular \cm{matrix
polynomials with symmetry structures of} arbitrary degree. The generic change of the partial multiplicities of regular \cm{matrix polynomials
without additional symmetry structures}
has been described in \cite{dd09}. However, the case of \cm{structure-preserving} perturbations of \cm{matrix polynomials
with symmetry structures} remains open.
\end{itemize}

\appendix

\section{Appendix}\label{appendix}
This appendix is devoted to prove \cm{the identities~\eqref{eq:11.1.19} and \eqref{det-identity}}.

\cm{We start with~\eqref{eq:11.1.19}. In this case $n_r$ is odd, say $n_r=2k+1$. The case $k=0$ is straightforward, so
we assume $k>0$.
Setting $\Delta_k:=\det(R\diag J_{2k+1}(-\la),(J_{2k+1}(\la))+\gamma (e_1+e_{2k+3})(e_1+e_{2k+3})^\top$ we have
\[
\Delta_k=\left|
\begin{array}{cccc|cccc}
\gamma&0&\hdots&0&0&\gamma&&\la\\
0&0&\hdots&0&&&\la&1\\
\vdots&\ddots&\ddots&\vdots&&\iddots&\iddots\\
0&0&\hdots&0&\la&1&\\\hline
0&\hdots&0&-\la&0&0&\hdots&0\\
\gamma&&-\la&1&0&\gamma&\hdots&0\\
&&\iddots\iddots&&\vdots&&&\vdots\\
-\la&1&&&0&0&\hdots&0
\end{array}
\right|.
\]
Using the Laplace expansion with respect to  the $(2k+1)$st row and column we arrive at
\[
\Delta_k=\la^2\cdot\left|
\begin{array}{cccc|cccc}
\gamma&0&\hdots&0&\gamma&&&\la\\
0&0&\hdots&0&&&\la&1\\
\vdots&\vdots&\ddots&\vdots&&\iddots&\iddots\\
0&0&\hdots&0&\la&1&\\\hline
\gamma&&&-\la&\gamma&0&\hdots&0\\
&&-\la&1&0&0&\hdots&0\\
&\iddots\iddots&&&\vdots&\vdots&\ddots&\vdots\\
-\la&1&&&0&0&\hdots&0
\end{array}
\right|.
\]
Using the Laplace expansion with respect to  the last column, we obtain
\begin{equation}\label{deltak2}\footnotesize
\Delta_k=\la^2\left((-\la)
\left|\begin{array}{cccc|cccc}
&&&&&&&\la\\
&&&&&&\la&1\\
&&&&&\iddots&\iddots\\
&&&&\la&1\\\hline
\gamma&&&-\la&\gamma&&\\
&&-\la&1&&&&\\
&\iddots&\iddots&&&&\\
-\la&1&&&&&&
\end{array}\right|+
\left|\begin{array}{cccc|cccc}
\gamma&&&&\gamma&&&0\\
&&&&&&\la&1\\
&&&&&\iddots&\iddots\\
&&&&\la&1\\\hline
\gamma&&&-\la&\gamma&&\\
&&-\la&1&&&&\\
&\iddots&\iddots&&&&\\
-\la&1&&&&&&
\end{array}\right|
\right).
\end{equation}
Computing  separately the first and second determinant again via Laplace expansion,
the first determinant is equal to
\[
\begin{array}{ccl}
(-1)^{k-1}\la^{2k-1}\left|\begin{array}{cccc}\gamma&&&-\la\\&&-\la&1\\&\iddots&\iddots\\
-\la&1&&\end{array}\right|&=&
(-1)^{k-1}\la^{2k-1}\left((-1)^{k-1}\gamma-(-1)^{k-1}(-\la)^{2k}\right)\\
&=&\la^{2k-1}\left(\gamma-\la^{2k}\right),
\end{array}
\]
%
and the second determinant is
\[
\begin{array}{ccl}
(-1)^{k-1}\left|\begin{array}{ccccc}\gamma&0&\hdots&0&\gamma\\
\gamma&0&\hdots&-\la&\gamma\\
&&\iddots&1&\\
&-\la&\iddots&&\\
-\la&1&&&
\end{array}\right|&=&(-1)^{k-1}\left|\begin{array}{ccccc}\gamma&0&\hdots&0&\gamma\\
&&&-\la&\\
&&\iddots&1&\\
&-\la&\iddots&&\\
-\la&1&&&
\end{array}\right|\\&=&(-1)^k\gamma
\left|\begin{array}{cccc}
&&&-\la\\
&&-\la&1\\
&\iddots&\iddots&\\
-\la&1&&
\end{array}\right|=-\gamma\la^{2k}.
\end{array}
\]
so that for
\eqref{deltak2} we get
\[
\Delta_k=-\la^2\left(-\la^{2k}(\gamma-\la^{2k})-\gamma\la^{2k}\right)=\la^{2k+2}\left(\la^{2k}-2\gamma\right),
\]
as claimed.}

\cm{The proof of~\eqref{det-identity} proceeds analogously, with only minor modifications. Now $n_r$ is even, say $n_r=2k$. Thus, setting}
$\cm{\widetilde\Delta_k}:=\det(R\cm{\diag (-J_{2k}(-\la),J_{2k}(\la)})+\gamma \lambda(e_1+e_{2k+2})(e_1+e_{2k+2})^\top$ we have
\[
\cm{\widetilde\Delta_k}=\left|
\begin{array}{cccc|cccc}
\gamma\la&0&\hdots&0&0&\gamma\la&&\la\\
0&0&\hdots&0&&&\la&1\\
\vdots&\ddots&\ddots&\vdots&&\iddots&\iddots\\
0&0&\hdots&0&\la&1&\\\hline
0&\hdots&0&\la&0&0&\hdots&0\\
\gamma\la&&\la&-1&0&\gamma\la&\hdots&0\\
&&\iddots\iddots&&\vdots&&&\vdots\\
\la&-1&&&0&0&\hdots&0
\end{array}
\right|.
\]
Using the Laplace expansion with respect to  the $(2k+1)$st row and column we arrive at
\[
\cm{\widetilde\Delta_k}=-\la^2\cdot\left|
\begin{array}{cccc|cccc}
\gamma\la&0&\hdots&0&\gamma\la&&&\la\\
0&0&\hdots&0&&&\la&1\\
\vdots&\vdots&\ddots&\vdots&&\iddots&\iddots\\
0&0&\hdots&0&\la&1&\\\hline
\gamma\la&&&\la&\gamma\la&0&\hdots&0\\
&&\la&-1&0&0&\hdots&0\\
&\iddots\iddots&&&\vdots&\vdots&\ddots&\vdots\\
\la&-1&&&0&0&\hdots&0
\end{array}
\right|.
\]
Using the Laplace expansion with respect to  the last column, the previous expression is equal to
\begin{equation}\label{deltak}\footnotesize
\cm{\widetilde\Delta_k}=-\la^2\left((-\la)
\left|\begin{array}{cccc|cccc}
&&&&&&&\la\\
&&&&&&\la&1\\
&&&&&\iddots&\iddots\\
&&&&\la&1\\\hline
\gamma\la&&&\la&\gamma\la&&\\
&&\la&-1&&&&\\
&\iddots&\iddots&&&&\\
\la&-1&&&&&&
\end{array}\right|+
\left|\begin{array}{cccc|cccc}
\gamma\la&&&&\gamma\la&&&0\\
&&&&&&\la&1\\
&&&&&\iddots&\iddots\\
&&&&\la&1\\\hline
\gamma\la&&&\la&\gamma\la&&\\
&&\la&-1&&&&\\
&\iddots&\iddots&&&&\\
\la&-1&&&&&&
\end{array}\right|
\right).
\end{equation}
Computing  separately the first and second determinant again via Laplace expansion,
the first determinant is equal to
\[
\begin{array}{ccl}
(-1)^{k-1}\la^{2k-2}\left|\begin{array}{cccc}\gamma\la&&&\la\\&&\la&-1\\&\iddots&\iddots\\
\la&-1&&\end{array}\right|&=&
(-1)^{k-1}\la^{2k-2}\left((-1)^{k-1}\gamma\la+(-1)^{k-1}\la^{2k-1}\right)\\
&=&\la^{2k-1}\left(\la^{2k-2}+\gamma\right),
\end{array}
\]
%
and the second determinant is
\[
\begin{array}{ccl}
(-1)^{k-1}\left|\begin{array}{ccccc}\gamma\la&0&\hdots&0&\gamma\la\\
\gamma\la&0&\hdots&\la&\gamma\la\\
&&\iddots&-1&\\
&\la&\iddots&&\\
\la&-1&&&
\end{array}\right|&=&(-1)^{k-1}\left|\begin{array}{ccccc}\gamma\la&0&\hdots&0&\gamma\la\\
&&&\la&\\
&&\iddots&-1&\\
&\la&\iddots&&\\
\la&-1&&&
\end{array}\right|\\&=&(-1)^k\gamma\la
\left|\begin{array}{cccc}
&&&\la\\
&&\la&-1\\
&\iddots&\iddots&\\
\la&-1&&
\end{array}\right|=-\gamma\la^{2k}.
\end{array}
\]
so that for
\eqref{deltak} we get
\[
\cm{\widetilde\Delta_k}=-\la^2\left(-\la^{2k}(\la^{2k-2}+
\gamma)-\gamma\la^{2k}\right)=\la^{2k+2}\left(\la^{2k-2}+2\gamma\right),
\]
as claimed.

\bigskip

\noindent{\bf Acknowledgments.} The work of Fernando De Ter\'{a}n has been supported by the Ministerio de Econom\'{i}a y Competitividad of Spain
through grants MTM2015-68805-REDT, and MTM2015-65798-P, and by the Ministerio de Educaci\'{o}n, Cultura y Deportes of Spain through grant PRX16/00128 ``Programa de estancias de movilidad de profesores e investigadores en centros de ense\~{n}anza superior e investigaci\'{o}n ``Salvador de Madariaga"".

The work of Volker Mehrmann has been supported by Einstein Foundation Berlin through project OT3 within the Einstein Center ECMath.
\bibliographystyle{plain}
\bibliography{biblio}

\end{document}